\documentclass[reqno]{amsart}
 
\usepackage{amsmath}
\usepackage{amsthm,amssymb,color,comment}
\usepackage{bbm}
\usepackage{mathrsfs}
\usepackage{hyperref}
\usepackage[TS1,T1]{fontenc}
\usepackage[utf8]{inputenc}
\usepackage{dsfont}
\usepackage{tikz}
\usepackage{enumitem}
\usepackage{subfigure}
\usepackage{mathtools}
\usepackage{bbold}
\usepackage{esint}
\usepackage[sort]{cite}

\numberwithin{equation}{section}

\newtheorem{thm}{Theorem}[section]
\newtheorem{proposition}[thm]{Proposition}
\newtheorem{lem}[thm]{Lemma}

\newtheorem{Def}[thm]{Definition}

\theoremstyle{definition}

\newtheorem{rem}[thm]{Remark}

\DeclareMathOperator{\Id}{Id}

\DeclareMathOperator{\DIV}{div}

\DeclareMathOperator*{\argmin}{argmin}
\DeclareMathOperator{\Dom}{Dom}

\newcommand{\bS}{\mathbf{S}}
\newcommand{\W}{\mathcal{W}}
\newcommand{\R}{\mathbb{R}}
\newcommand{\Td}{\mathbb{T}^{d}}
\newcommand{\Rd}{\mathbb{R}^{d}}
\newcommand{\Pro}{\mathscr{P}(\Td)}
\newcommand{\N}{\mathbb{N}}
\newcommand{\p}{\partial}
\newcommand {\f}{\frac}
\newcommand{\eps}{\varepsilon}

\newcommand{\F}{\mathscr{F}}

\newcommand{\supp}{\text{supp}}
\newcommand{\diff}{\mathop{}\!\mathrm{d}}
\newcommand{\ie}{\emph{i.e.}\;}

\newcommand{\ueps}{u_{\varepsilon}}

\newcommand{\weak}{\rightharpoonup}
\DeclarePairedDelimiter{\norm}{\lVert}{\rVert}

\makeatletter
\newcommand{\doublewidetilde}[1]{{%
  \mathpalette\double@widetilde{#1}%
}}
\newcommand{\double@widetilde}[2]{%
  \sbox\z@{$\m@th#1\widetilde{#2}$}%
  \ht\z@=.9\ht\z@
  \widetilde{\box\z@}%
}
\makeatother
\author{José A. Carrillo}
\address{{\it José A. Carrillo:} Mathematical Institute, University of Oxford, Woodstock Road, Oxford, OX2 6GG, United Kingdom}
\email{carrillo@maths.ox.ac.uk}

\author{Charles Elbar}
\address{{\it Charles Elbar:} Sorbonne Universit\'{e}, Laboratoire Jacques-Louis Lions (LJLL), F-75005 Paris, France}
\email{charles.elbar@sorbonne-universite.fr}

\author{Jakub Skrzeczkowski}
\address{{\it Jakub Skrzeczkowski: } Institute of Mathematics of Polish Academy of Sciences; Faculty of Mathematics, Informatics and Mechanics, University of Warsaw, Poland}
\email{jakub.skrzeczkowski@student.uw.edu.pl}

\begin{document}

\title[]{Degenerate Cahn-Hilliard systems: From nonlocal to local}

\begin{abstract}
We provide a rigorous mathematical framework to establish the limit of a nonlocal model of cell-cell adhesion system to a local model. When the parameter of the nonlocality goes to 0, the system tends to a Cahn-Hilliard system with degenerate mobility and cross interaction forces. Our analysis relies on a priori estimates and compactness properties. 
\end{abstract}



\maketitle
\setcounter{tocdepth}{1}

\section{Introduction}

\noindent We consider the nonlocal system of Cahn-Hilliard equation with degenerate mobility derived in~\cite{falco2022local}
\begin{align}
     \frac{\partial \rho}{\partial t}& = \nabla\cdot\left(\rho\nabla \left(\kappa B[\rho] + \alpha B[\eta] - \gamma\rho - \beta\eta \right)\right), &&\text{in}\quad (0,+\infty)\times \Td, \label{eq:falco_eq1}\\
     \frac{\partial \eta}{\partial t}& = \nabla\cdot\left(\eta\nabla \left(\alpha B[\rho] + B[\eta] - \beta\rho - \eta \right)\right),&&\text{in}\quad (0,+\infty)\times \Td. \label{eq:falco_eq2}
\end{align}

equipped with an initial datum $(\rho_{0},\eta_{0})\in\Pro^{2}$ where $\Pro$ is the space of probability measure on the flat torus of dimension $d$. Here $\kappa, \alpha, \gamma > 0$, $\beta \in \R$ while $B$ is the nonlocal operator $B=B_{\varepsilon}$ defined with
\begin{equation}\label{operatorB}
B_\eps[u_{\eps}](x) = \f{1}{\eps^{2}}(u_{\eps}(x)-\omega_{\eps}\ast u_{\eps}(x))=\f{1}{\eps^{2}}\int_{\Td}\omega_{\eps}(y)(u_{\eps}(x)-u_{\eps}(x-y)) \diff y
\end{equation}
for $\eps$ small enough and $\omega_{\eps}$ is the usual radial mollification kernel $\omega_{\eps}(x)=\frac{1}{\eps^{d}}\omega(\frac{x}{\eps})$ with $\omega$ compactly supported in the unit ball satisfying
\begin{equation}
\int_{\Rd}\omega(y) \diff y =1, \quad \int_{\Rd} y\, \omega(y) \diff y =0,  \quad \int_{\Rd}   y_i y_j \omega \diff y = \delta_{i,j}\f{2}{d}.
\label{as:omega}
\end{equation}

 Our target is to prove that as $\varepsilon \to 0$, the constructed solutions of 
\begin{align}
     \frac{\partial \rho}{\partial t}& = \nabla\cdot\left(\rho\nabla \left(\kappa B_\eps[\rho] + \alpha B_\eps[\eta] - \gamma\rho - \beta\eta \right)\right), &&\text{in}\quad (0,+\infty)\times \Td,\label{eq:CHNL1}\\
     \frac{\partial \eta}{\partial t}& = \nabla\cdot\left(\eta\nabla \left(\alpha B_\eps[\rho] + B_\eps[\eta] - \beta\rho - \eta \right)\right),&&\text{in}\quad (0,+\infty)\times \Td\label{eq:CHNL2}
\end{align}
tend to the weak solution of the local system of degenerate Cahn-Hilliard equations
\begin{align}
     \frac{\partial \rho}{\partial t}& = \nabla\cdot\left(\rho\nabla \left(-\kappa\Delta \rho -\alpha\Delta \eta - \gamma\rho - \beta\eta \right)\right),\label{eq:CHL1}\\
     \frac{\partial \eta}{\partial t}& = \nabla\cdot\left(\eta\nabla \left(-\alpha\Delta \rho - \Delta \eta - \beta\rho - \eta \right)\right)\label{eq:CHL2}.
\end{align}
\noindent Here, $\kappa>0$ and $\gamma> 0$ represent the relative self-adhesion strength of $\rho$ with respect to $\eta$; while $\alpha> 0$ and $\beta\in\mathbb{R}$ give the relative strength of the cross-attraction forces. We also denote by
\begin{align*}
\mu_{\rho,\eps}&=\kappa B_\eps[\rho] + \alpha B_\eps[\eta] - \gamma\rho - \beta\eta\,,\\
\mu_{\eta,\eps}&=\alpha B_\eps[\rho] + B_\eps[\eta] - \beta\rho - \eta\,,
\end{align*}
the chemical potentials (from the Cahn-Hilliard terminology) related to the nonlocal system. 
The nonlocal system is associated with the following formal energy/entropy structure
\begin{equation}\label{eq:intro_energy}
\begin{split}
 E_{\varepsilon}&[\rho,\eta]:=\frac{1}{4}\int_{\Td}\int_{\Td}\f{\omega_{\eps}(y)}{\eps^{2}} (\kappa|\rho(x)-\rho(x-y)|^{2}+|\eta(x)-\eta(x-y)|^{2})\diff x\diff y  \\ &+\f{\alpha}{2}\int_{\Td}\int_{\Td}\frac{\omega_{\eps}(y)}{\eps^{2}}(\rho(x)-\rho(x-y))(\eta(x)-\eta(x-y))\diff x\diff y -\int_{\Td}\f{\gamma}{2}\rho^{2}+\f{1}{2}\eta^{2}+\beta\rho\eta\diff x,
 \end{split}
 \end{equation}
 \begin{equation}
\Phi[\rho,\eta]:=\int_{\Td}\rho(\log(\rho)-1)+\eta(\log(\eta)-1) \diff x \label{eq:intro_entropy}.
 \end{equation}
 
 Their dissipation is formally controlled by the identities
\begin{equation}\label{eq:energy_diss}
E_{\varepsilon}[\rho,\eta](t) + \int_{0}^{t}\int_{\Td} \rho \, |\nabla \mu_{\rho,\eps}|^2 + \int_{0}^{t}\int_{\Td} \eta \, |\nabla \mu_{\eta,\eps}|^2 \le E_{\varepsilon}[\rho_{0},\eta_{0}],
\end{equation}
\begin{equation}\label{eq:entropy_diss}
\begin{split}
\Phi[\rho,\eta](t) + \mathcal{D}\Phi[\rho,\eta](t) \le \Phi[\rho_{0},\eta_{0}].
\end{split}
\end{equation}
where $\mathcal{D} \Phi[\rho,\eta](t)$ is the dissipation of the entropy defined as 
\begin{align*}
\mathcal{D} \Phi[\rho,\eta](t) =&
\f{1}{2}\int_0^t\int_{\Td}\int_{\Td}\frac{\omega_{\eps}(y)}{\eps^{2}}(\kappa|\nabla\rho(x)-\nabla\rho(x-y)|^{2}+|\nabla\eta(x)-\nabla\eta(x-y)|^{2})\diff x\diff y \diff s \\
&+\alpha \int_0^t\int_{\Td}\int_{\Td}\f{\omega_{\eps}(y)}{\eps^{2}}(\nabla\rho(x)-\nabla\rho(x-y))\cdot(\nabla\eta(x)-\nabla\eta(x-y))\diff x\diff y \diff s\\
&-\int_0^t\int_{\Td}\gamma|\nabla\rho|^{2}+|\nabla\eta|^{2}+2\beta\nabla\rho\cdot\nabla\eta \diff x \diff s\,.
\end{align*}

It turns out (see Proposition \ref{est:positive_energy_entropy}) that for $\kappa > \alpha^2$, $\kappa > 0$ they provide strong compactness in space of $\rho_{\eps}$, $\eta_{\eps}$, $\nabla \rho_{\eps}$, $\nabla \eta_{\eps}$.\\

Our first result states that we can construct solutions to~\eqref{eq:CHNL1}-\eqref{eq:CHNL2} satisfying additional uniform estimates which will be relevant in the sequel.
\begin{thm}[Existence of solutions for the nonlocal system]\label{thm:weaksoldelta}
Suppose that $\kappa > 0$, $\kappa > \alpha^2$. Let $\varepsilon_0 = \min(\varepsilon_0^A, \varepsilon_0^B)$ be given by Proposition~\ref{est:positive_energy_entropy} and Lemma \ref{lem:convexity}. Let $u_{0}=(\rho_{0},\eta_{0})\in\Pro^{2}$ be an initial datum with finite energy and entropy $E_{\eps}(u^0), \Phi(u^0)\le C$ defined in ~\eqref{eq:intro_energy}-\eqref{eq:intro_entropy} where $C$ is independent of $\eps$.  Let $\varepsilon \le \varepsilon_0$.  Then, there exists a global weak solution $\ueps \in\Pro^{2}$ of~\eqref{eq:CHNL1}-\eqref{eq:CHNL2} as defined in Definition~\ref{def:weak_sol_2}. Moreover, it satisfies
\begin{equation}\label{eq:thm_regularity_compactness_rho_eta}
\int_{\Td}\int_{\Td}\f{\omega_{\eps}(y)}{\eps^{2}}(|\rho_\eps(x)-\rho_\eps(x-y)|^{2}+|\eta_\eps(x)-\eta_\eps(x-y)|^{2})\diff x\diff y\le C,
\end{equation}
\begin{equation}\label{eq:thm_regularity_compactness_nabla_rho_eta}
\int_0^T \int_{\Td}\int_{\Td}\f{\omega_{\eps}(y)}{\eps^{2}} (|\nabla\rho_\eps(x)-\nabla\rho_\eps(x-y)|^{2}+|\nabla\eta_\eps(x)-\nabla\eta_\eps(x-y)|^{2})\diff x\diff y \diff t \le C,
\end{equation}
for a constant $C$ that depends on parameters and the initial condition $u^0$ but not on $\eps$.
\end{thm}

\noindent The proof of this result follows the argument of \cite{MR3761096}. We decided to include the proof to demonstrate estimates \eqref{eq:thm_regularity_compactness_rho_eta}--\eqref{eq:thm_regularity_compactness_nabla_rho_eta} which are essential for our main result which reads as follows.

\begin{thm}[Convergence of nonlocal to local Cahn-Hilliard equation on the torus]\label{thm:final}
Suppose that $\kappa > 0$, $\kappa > \alpha^2$. Let $u_{0}\ge 0$ be an initial datum with finite energy and entropy $E_{\eps}(u^0), \Phi(u^0)\le C$ defined in ~\eqref{eq:intro_energy} and \eqref{eq:intro_entropy} where $C$ is independent of $\eps$. Let $\{u_{\varepsilon}\}$ be a sequence of solutions of the degenerate nonlocal Cahn-Hilliard equation~\eqref{eq:CHNL1}-\eqref{eq:CHNL2} as defined in Definition \ref{def:weak_sol_2}. Then, up to a subsequence (not relabelled),
$$
u_{\varepsilon} \to u \mbox{ in } L^2(0,T; H^1(\Td))\,,
$$
where $u$ is a weak solution of the degenerate Cahn-Hilliard system~\eqref{eq:CHL1}-\eqref{eq:CHL2} as in Definition~\ref{def:weak_sol_local}.   
\end{thm}

We conclude with a short discussion of applied techniques and considered problems.

\underline{\textit{Existence of weak solutions.}}
The strategy to prove the existence of weak solutions is based on the gradient flow structure of the equation in the Wasserstein space. We use the JKO scheme first introduced in~\cite{MR1617171} by Jordan, Kinderlehrer and Otto in the context of the Fokker-Planck equation. The main idea is to use an implicit time discretization of the associated variational problems. The sequence created minimizes movements. When the time step goes to 0, the sequence converges to the associated gradient flow. This strategy has been used for instance in~\cite{MR2921215,kroemer-laux} in the context of Cahn-Hilliard equation and in~\cite{MR3761096} in the context of cross-diffusion systems with nonlocal interaction. The proof of Theorem \ref{thm:weaksoldelta} is in fact closely related to the one in~\cite{MR3761096} and we follow to some extent their proof. The difference is the control of the positiveness of the energy and the transport of uniform bounds independently of $\eps$ needed in the last section. Also, our settings are periodic in space.
The proof uses two main components:
\begin{itemize}
    \item The "gradient flow" structure of the scheme which provides classical discrete energy estimates and Hölder continuity in time and allow narrow convergences of the scheme.
    \item The weak convergences being not enough, we use the flow interchange lemma~\cite{MR2581977,MR2921215}. The idea is that we usually obtain better estimates with the entropy in the Cahn-Hilliard equation. This entropy generates a heat-flow. Then the flow interchange lemma allows to exchange the dissipation of one functional along the gradient flow of another one and thus improve the regularity of the scheme. 
\end{itemize}
Nevertheless, since we assume that $\kappa > \alpha^2$ and $\kappa > 0$, the system is strongly parabolic with respect to $(\rho, \eta)$. Therefore, to prove existence, one could approximate \eqref{eq:CHNL1}--\eqref{eq:CHNL2} by
\begin{align}
     \frac{\partial \rho}{\partial t}& = \nabla\cdot\left(T_{\delta}(\rho)\nabla \left(\kappa B_\eps[\rho] + \alpha B_\eps[\eta] - \gamma\rho - \beta\eta \right)\right), &&\text{in}\quad (0,+\infty)\times \Td,\label{eq:CHNL1_delta}\\
     \frac{\partial \eta}{\partial t}& = \nabla\cdot\left(T_{\delta}(\eta)\nabla \left(\alpha B_\eps[\rho] + B_\eps[\eta] - \beta\rho - \eta \right)\right),&&\text{in}\quad (0,+\infty)\times \Td,\label{eq:CHNL2_delta}
\end{align}
where $T_{\delta}$ is a function such that $\delta \leq T_{\delta} \leq \frac{1}{\delta}$ and $T_{\delta}(\varrho) \to \varrho$ as $\delta \to 0$. Then, one sends $\delta \to 0$ and obtains solutions to \eqref{eq:CHNL1}--\eqref{eq:CHNL2}. This method is a standard way of proving existence of solutions to Cahn-Hilliard equation with degenerate mobility \cite{MR3448925,MR1377481,elbar2021degenerate,MR4199231}. Of course, one has to prove estistence to \eqref{eq:CHNL1_delta}--\eqref{eq:CHNL2_delta} by virtue of fixed point method and Schauder's estimates for parabolic equations which is very technical. Therefore, we preferred to apply gradient flow techniques which are natural for our problem.

\underline{\textit{Passage to the limit $\eps\to 0$ and nonlocal compactness results.}}
We use the strategy developed by the second and third author in~\cite{elbar-skrzeczkowski} for the single Cahn-Hilliard equation. The main tool is the compactness result due to Bourgain-Brezis-Mironescu \cite{bourgain2001another} and Ponce \cite{MR2041005} which reads as follows:
\begin{proposition}\label{prop:ponce}
Let $\{f_\eps\}$ be a sequence bounded in $L^2(\Td)$. Suppose that
\begin{equation}\label{eq:compactness_theorem_ponce}
\int_{\Td} \int_{\Td} \frac{|f_\eps(x) - f_\eps(y)|^2}{|x-y|^2} \omega_\eps(|x-y|) \diff x \diff y \leq C
\end{equation}
for some constant $C$. Then, $\{f_\eps\}$ is strongly compact in $L^2(\Td)$ and the limit $f \in W^{1,2}(\Td)$.
\end{proposition}
Proposition \ref{prop:ponce} is crucial because the difference quotient appearing in \eqref{eq:compactness_theorem_ponce} can be controlled by the energy and the dissipation of the entropy of Cahn-Hilliard equation which yields compactness. The novelty in this paper comes from the treating of cross-interactions terms. It turns out that we can control possibly negative terms appearing in the energy and the dissipation of entropy by a simple interpolation inequality, see Lemma \ref{lem:poincare_with_parameter} and Proposition \ref{est:positive_energy_entropy}. 

\underline{\textit{Aggregation-diffusion systems.}} System \eqref{eq:CHNL1}-\eqref{eq:CHNL2} is an example of aggregation-diffusion system \cite{MR3932458} which are attracting a lot of mathematical attention nowadays \cite{MR4344434,MR4223480,MR4193178,MR4022083,MR3847183}. To motivate, let us first start with an aggregation equation
\begin{equation}\label{eq:aggregation_equation}
\frac{\partial \rho}{\partial t} + \nabla\cdot (\rho u) = 0, \qquad u = -\nabla W \ast \rho,
\end{equation}
where $W:\Td \to \R$ is a symmetric interaction potential. Over the last years, equation \eqref{eq:aggregation_equation} was applied in the context of biological aggregation \cite{MR1794944,TBL06,MR1698215}, materials science \cite{Holm_Putkaradze}, granular media \cite{MR2209130, MR1471181} and it attracted also a lot of mathematical interest, particularly for non-smooth potentials. We only refer to \cite{MR2743876} for the $L^p$ theory, to \cite{MR2769217} for the theory in spaces of measures and to \cite{MR2480108} for blow-up conditions. Equation \eqref{eq:aggregation_equation} can be derived from the particle system via mean-field limit \cite{MR3558251, MR3317577}: one considers $N$ particles having positions $\{X_i\}_{i = 1,...,N}$ satisfying system of ODEs:
$$
X_i'  = - \frac{1}{N} \sum_{j \neq i} \nabla W(X_i - X_j);
$$
then, under appropriate assumptions on $W$, empirical measure $\rho^N(t) = \frac{1}{N} \sum_{i=1}^N \delta_{X_i(t)}$ converges (in the weak$^*$ topology of measures) to a solution of \eqref{eq:aggregation_equation}, see \cite[Theorem 3.1]{MR3331178} for the proof and \cite{MR3880206} for some extensions.\\

\noindent To obtain the aggregation-diffusion equation from \eqref{eq:aggregation_equation}, one can consider interaction potentials of the form $W_{\nu} = W + 2\nu \delta_0$ modeling two effects: repulsion of strength $\nu$ and non-local attraction. Informally, we obtain
\begin{equation}\label{eq:aggr_diff_m2}
\frac{\partial \rho}{\partial t} = \nu \, \Delta \rho^2 + \nabla \cdot(\rho\, \nabla W \ast \rho),
\end{equation}
but this can be made rigorous by approximating the Dirac mass with a sequence of mollifiers, see \cite{MR3913840,burger2022porous} with $m=2$. Of course, \eqref{eq:aggr_diff_m2} can be extended to include more general diffusion term $\Delta \rho^m$ instead of $\Delta \rho^2$ as well as an advection term, see \cite{carrillo2023nonlocal}.\\

One can also consider system of particles $\{X_i\}_{i=1,...,N}$, $\{Y_i\}_{i=1,...,N}$ representing two populations. Then, repeating the derivation explained above, one arrives at a system of equations of the form \eqref{eq:falco_eq1}--\eqref{eq:falco_eq2} which can be used to model cell-cell adhesion \cite{ArmstrongPainterSherratt,CHS18,MR3948738} to reproduce the Steinberg cell-sorting phenomena. Structure preserving numerical schemes have been derived for equations and systems of aggregation-diffusion type \cite{CCH15,B.C.H2020,bailo2023boundpreserving} as well as Cahn-Hilliard equations and systems \cite{bailo2023unconditional,falco2022local} recovering the cell sorting mechanism.\\

\underline{\textit{Gradient flows in the periodic setting.}} Due to the presence of a non-local operator $B_{\varepsilon}$ in \eqref{eq:CHNL1}--\eqref{eq:CHNL2}, we develop our theory on the $d$-dimensional torus $\Td$ which makes non-local terms easy to be defined. As already explained, solutions to \eqref{eq:CHNL1}--\eqref{eq:CHNL2} will be constructed via JKO scheme. In what follows, we briefly review the theory of optimal transport on $\Td$ comparing to the usual case of $\Rd$ or bounded domain $\Omega \subset \R^d$. We refer the Reader to \cite[Section 1.3.2]{MR3409718}.\\

First, we need to define a metric and the natural choice is 
$$
d(x,y) = \inf_{k \in \mathbb{Z}^d} |x- y +k |, \qquad x, y \in \Td,
$$
where $|\cdot|$ is the Euclidean distance. Then, the Wasserstein distance ${W}_{2}$ is defined as
$$
{W}_{2}^{2}(\mu, \nu) = \inf_{\pi \in \Pi(\mu, \nu)} \int_{\Td \times \Td} d(x,y)^2 \diff \pi(x,y), \qquad \mu, \nu \in \Pro,
$$
where $\Pi(\mu, \nu)$ is the set of couplings between $\mu$ and $\nu$. This already implies that, say in dimension $d=1$, the optimal transport maps are not necessarily monotone on the torus. Nevertheless, the optimal map always exists if at least one of the measures is absolutely continuous with respect to the Lebegue measure. Moreover, the optimal map is given by the gradient of some function which is differentiable a.e. We refer to \cite{MR1711060} for the first proof of this fact and to \cite[Theorem 1.25]{MR3409718} for a modern presentation. We also refer to the general result of McCann \cite{MR1844080} who proved existence of the optimal map on the general manifold, including the case of torus.\\

To conclude, let us mention that optimal transport was used to study several PDEs on the torus via JKO scheme, in particular fractional porous medium equation \cite{MR3912487}, continuity equation with nonlocal velocity in 1D \cite{MR3473438}, systems of continuity equations with nonlinear diffusion and nonlocal drifts \cite{carlier:hal-01147666} and certain fourth-order equation in one dimension \cite{MR2558330}.

\section{Energy and entropy for the non-local system}

As already mentioned, system \eqref{eq:CHNL1}--\eqref{eq:CHNL2} has the energy/entropy structure which can be used to obtain compactness estimates. In order to do so, one has to assure nonnegativity of the quantities of interest. In our case, we focus on the energy $E_{\eps}$ and the dissipation of entropy $\mathcal{D}\Phi$.

\begin{proposition}\label{est:positive_energy_entropy}
Suppose that $\kappa > 0$, $\kappa > \alpha^2$. Then, there exists $\varepsilon_0^A>0$ depending on $\kappa$, $\alpha$, $\beta$, $\gamma$ with the following property: for all $\varepsilon \in (0,\varepsilon_0^A)$ and $\rho, \eta \in \Pro$, up to a constant, the energy defined by~\eqref{eq:intro_energy} and the dissipation of the entropy defined in~\eqref{eq:entropy_diss} are nonnegative and provide the estimates on the quantities
\begin{align*}
&\int_{\Td}\int_{\Td}\f{\omega_{\eps}(y)}{\eps^{2}}(|\rho(x)-\rho(x-y)|^{2}+|\eta(x)-\eta(x-y)|^{2})\diff x\diff y\le C + E_{\varepsilon}[\rho,\eta],\\
&\int_0^t \int_{\Td}\int_{\Td}\f{\omega_{\eps}(y)}{\eps^{2}} (|\nabla\rho(x)-\nabla\rho(x-y)|^{2}+|\nabla\eta(x)-\nabla\eta(x-y)|^{2})\diff x\diff y \diff s \le C + \mathcal{D}\Phi[\rho,\eta](t),
\end{align*}
where $C$ depends on $\kappa$, $\alpha$, $\beta$ and $\gamma$.
\end{proposition}

The main tool to establish nonnegativity (up to a constant) is the following non-local Poincare inequality with parameter which allows to handle negative terms.

\begin{lem}[Poincare inequality with a parameter]\label{lem:poincare_with_parameter}
For all $\delta > 0$ there exists $C(\delta)$ and $\varepsilon_0 = \varepsilon_0(\delta)$ such that for all $\varepsilon < \varepsilon_0$ and all $f \in H^1(\Td)$ we have
\begin{equation}\label{eq:Poincare_f}
\| f \|_{2}^2 \leq \delta \, \int_{\Td\times\Td}\omega_{\eps}(y)\f{| f(x)-  f(x-y)|^{2}}{\eps^{2}}\diff x\diff y + C(\delta) \|f\|_{1}^2,
\end{equation}
\begin{equation}\label{eq:Poincare_grad_f}
\| \nabla f \|_{2}^2 \leq \delta \, \int_{\Td\times\Td}\omega_{\eps}(y)\f{|\nabla f(x)- \nabla f(x-y)|^{2}}{\eps^{2}}\diff x\diff y + C(\delta) \|f\|_{1}^2. 
\end{equation}
\end{lem}
\begin{proof}
We prove \eqref{eq:Poincare_grad_f} as \eqref{eq:Poincare_f} is in fact easier. Aiming at a contradiction, suppose that there exists $\delta>0$ with the following property: there exists sequence $\{\varepsilon_n\}$ with $0<\varepsilon_n < \frac{1}{n}$ and sequence $\{f_n\}$ such that
$$
\| \nabla f_{n}\|^2_{L^2(\Td)} > \delta \int_{\Td} \int_{\Td}  \frac{|\nabla f_{n}(x) - \nabla f_{
n}(y)|^2}{\eps_n^2} \omega_{\eps_n}(|x-y|) \diff x \diff y + n\, \|f_{n}\|^2_{L^1(\Td)}.
$$
As $\| \nabla f_{n}\|_{L^2(\Td)} > 0$, we may define $g_{n} := \frac{f_{n}}{\| \nabla f_{n}\|_{L^2(\Td)}}$. Note that $\|\nabla g_{n}\|_{L^2(\Td)} = 1$ and
$$
1 > \gamma \int_{\Td} \int_{\Td}  \frac{|\nabla g_{n}(x) - \nabla g_{n}(y)|^2}{\eps_n^2} \omega_{\eps_n}(|x-y|) \diff x \diff y + n\, \|g_{n}\|_{L^1(\Td)}^2.
$$
By Poincare inequality with average, $\{g_n\}$ is bounded in $L^2(\Td)$ and so in $H^ 1(\Td)$. Moreover, the first term gives compactness of the gradients (because $\{g_{n}\}$ is bounded in $H^1(\Td)$ so that, together with Rellich-Kondrachov, there exists function $g$ such that $g_{n} \to g$ in $H^1(\Td)$ (after passing to a subsequence). But then $g = 0$ because $n\, \|g_{n}\|_{L^1(\Td)} < 1$. This is however contradiction with $\|\nabla g\|_{L^2(\Td)} = \lim_{n \to \infty} \|\nabla g_n\|_{L^2(\Td)} = 1$.
\end{proof}

\begin{proof}[Proof of Proposition \ref{est:positive_energy_entropy}]
We first focus on the energy. We can estimate
$$
-\int_{\Td}\f{\gamma}{2}\rho^{2}+\f{1}{2}\eta^{2}+\beta\rho\,\eta  \geq -\f{\gamma+|\beta|}{2} \int_{\Td} \rho^{2} - \f{1+|\beta|}{2} \int_{\Td}\eta^{2}.
$$
Then, we use \eqref{eq:Poincare_f} in Lemma \ref{lem:poincare_with_parameter} with $\delta:= \delta/\max\left(\f{\gamma+|\beta|}{2}, \f{1+|\beta|}{2} \right)$ and $\delta >0$ to be chosen later (this also determines $\varepsilon_0^A=\varepsilon_0^A(\delta)$) so that we obtain for $\varepsilon \in (0,\varepsilon_0^A)$
\begin{multline*}
 -\int_{\Td}\f{\gamma}{2}\rho^{2}+\f{1}{2}\eta^{2}+\beta\rho\,\eta \geq \\ 
 \geq - \frac{\delta}{4} \, \int_{\Td\times\Td}\omega_{\eps}(y)\f{| \rho(x)-  \rho(x-y)|^{2}}{\eps^{2}}\diff x\diff y - \frac{\delta}{4} \, \int_{\Td\times\Td}\omega_{\eps}(y)\f{| \eta(x)-  \eta(x-y)|^{2}}{\eps^{2}}\diff x\diff y - C(\delta).
\end{multline*}
Therefore, using \eqref{eq:intro_energy}, we can bound energy as follows
\begin{multline*}
E_{\varepsilon}[\rho,\eta] \geq \f{1}{4\eps^{2}}\int_{\Td}\int_{\Td}\omega_{\eps}(y)((\kappa-\delta)\,|\rho(x)-\rho(x-y)|^{2}+(1-\delta)\,|\eta(x)-\eta(x-y)|^{2})\diff x\diff y  \\ +\f{\alpha}{2\eps^{2}}\int_{\Td}\int_{\Td}\omega_{\eps}(y)(\rho(x)-\rho(x-y))(\eta(x)-\eta(x-y))\diff x\diff y - C(\delta).
\end{multline*}
Now, by continuity, we choose $\delta$ so small so that $\kappa - \delta > 0$ and $(\kappa-\delta)(1-\delta)-\alpha^2 > 0$, i.e. so that the matrix $\begin{pmatrix}
\kappa - \delta & \alpha \\
\alpha & 1-\delta
\end{pmatrix}$ is positively defined. It follows that the assosciated quadratic form is bounded from below, that is there exists constant $C$ (in fact, this constant is the smallest eigenvalue of the matrix) such that
$$
E_{\varepsilon}[\rho,\eta] \geq \f{C}{\eps^{2}}\int_{\Td}\int_{\Td}\omega_{\eps}(y)(|\rho(x)-\rho(x-y)|^{2}+ \,|\eta(x)-\eta(x-y)|^{2})\diff x\diff y - C(\delta).
$$
The proof for the entropy is the same: this time we use \eqref{eq:Poincare_grad_f} in place of \eqref{eq:Poincare_f}.
\end{proof}

\section{Existence of weak solutions to the nonlocal problem}

To prove the existence of weak solutions for our system, we want to apply the JKO scheme, see~\cite{MR1617171}. As $\varepsilon$ is fixed in this Section, we write $(\rho,\eta)$ for the solution, instead of $(\rho_{\eps}, \eta_{\eps})$. 
\begin{Def}[Weak solutions]\label{def:weak_sol_2}
We say that $u=(\rho(\cdot),\eta(\cdot)):[0,+\infty)\to \Pro^{2}$ is a weak solution of~\eqref{eq:CHNL1}-\eqref{eq:CHNL2} with initial condition $(\rho_{0},\eta_{0})$ satisfying~\eqref{ass:initial_condition} if $\rho, \eta \in C([0,T]; \Pro)$, $\rho,\eta\in L^{2}(0,T;H^{1}(\Td))\cap L^{\infty}(0,T;L^{2}(\Td))$ for all $T>0$, and if for all $\varphi,\phi\in C_{c}^{\infty}([0,+\infty)\times\Td)$ we have
\begin{equation*}
\begin{split}
-\int_{0}^{\infty}\int_{\Td}\rho\p_{t}\varphi\diff x\diff t &-\int_{\Td}\rho_{0}\varphi(0)\diff x= -\kappa\int_{0}^{\infty}\int_{\Td} \rho\,\nabla B_{\eps}[\rho] \cdot\nabla\varphi\diff x\diff t\\
& - \alpha\int_{0}^{\infty}\int_{\Td}\rho\, \nabla B_{\eps}[\eta]\cdot\nabla\varphi \diff x\diff t -\int_{0}^{\infty}\int_{\Td}\rho (\gamma\nabla\rho-\beta\nabla\eta)\cdot\nabla\varphi\diff x\diff t,
\end{split}    
\end{equation*}
\begin{equation*}
\begin{split}
-\int_{0}^{\infty}\int_{\Td}\eta\p_{t}\phi\diff x\diff t &-\int_{\Td}\eta_{0}\phi(0)\diff x= -\int_{0}^{\infty}\int_{\Td} \eta\, \nabla B_{\eps}[\eta] \cdot \nabla\phi\diff x\diff t\\
&- \alpha\int_{0}^{\infty}\int_{\Td}\eta \nabla B_{\eps}[\rho]\cdot\nabla\phi \diff x\diff t -\int_{0}^{\infty}\int_{\Td}\eta(\nabla\eta-\beta\nabla\rho)\cdot\nabla\phi\diff x\diff t.
\end{split}    
\end{equation*}
\end{Def}
We first rewrite system of equations~\eqref{eq:CHNL1}-\eqref{eq:CHNL2} in the following form 
\begin{align}
     \frac{\partial \rho}{\partial t}& = \nabla\cdot\left(\rho\nabla \left(\tilde{\kappa}\rho-(\tilde{\kappa}+\gamma)\omega_\eps\ast\rho +\tilde{\alpha}\eta-(\tilde{\alpha}+\beta)\omega_\eps \ast \eta\right)\right), &&\text{in}\quad (0,+\infty)\times \Td,\label{eq:CHNLW1}\\
     \frac{\partial \eta}{\partial t}& = \nabla\cdot\left(\eta\nabla \left(\tilde{\alpha}\rho-(\tilde{\alpha}+\beta)\omega_\eps\ast\rho +\tilde{c}\eta-(\tilde{c}+1)\omega_\eps \ast \eta \right)\right),&&\text{in}\quad (0,+\infty)\times \Td.\label{eq:CHNLW2}
\end{align}
where $\tilde{\kappa}=\f{\kappa}{\eps^{2}}-\gamma$, $\tilde{\alpha}=\f{\alpha}{\eps^{2}}-\beta$, $\tilde{c}=\f{1}{\eps^{2}}-1$. Using this notation, the energy can be written as 
\begin{equation}\label{eq:energy_weak}
E_{\eps}[\rho,\eta]=\int_{\Td}\f{\tilde{\kappa}}{2}\rho^{2}+\f{\tilde{c}}{2}\eta^{2}+\tilde{\alpha}{\rho\eta}\diff x-\f{\tilde{\kappa}+\gamma}{2}\int_{\Td}\omega_{\eps}\ast\rho \diff\rho-\f{\tilde{c}+1}{2}\int_{\Td}\omega_{\eps}\ast\eta \diff\eta-(\tilde{\alpha}+\beta)\int_{\Td}\omega_{\eps}\ast\eta\diff\rho
\end{equation}
for $\rho,\eta\in\Pro$. We note that when we use $\rho$ and $\eta$ in expressions of the form $\rho^{2},\eta^{2},\rho\eta$, we mean in fact their Radon-Nikodym derivatives with respect to the Lebesgue measures. If they exist and belong to $L^{2}(\Td)$, the energy above makes sense. Otherwise, we consider that these quantities are $+\infty$.

We denote the Wasserstein distance for vectors $u=(u_{1},u_{2})$ and $v=(v_{1},v_{2})$ as  
\begin{equation*}
\W_{2}^{2}(u,v)=W_{2}^{2}(u_{1},v_{1})+W_{2}^{2}(u_{2},v_{2}),   
\end{equation*}
for all $u,v\in \Pro\times\Pro$ and where $W_{2}$ is the usual Wasserstein distance of order $2$. For the initial condition, we assume that $(\rho_{0},\eta_{0})\in\Pro^{2}$ are absolutely continuous with respect to the Lebesgue measure and satisfy 
\begin{equation}\label{ass:initial_condition}
E_{\eps}[\rho_{0},\eta_{0}]\le C, \quad \Phi[\rho_{0},\eta_{0}]\le C,    
\end{equation}
where $C$ is independent of $\eps$.

In what follows it will be necessary to know when the function $(\rho, \eta) \mapsto \f{\tilde{\kappa}}{2}\rho^{2}+\f{\tilde{c}}{2}\eta^{2}+\tilde{\alpha}{\rho\eta}$. Fortunately, when $\kappa >0$, $\kappa > \alpha^2$, this is always the case for sufficiently small $\varepsilon$.
\begin{lem}\label{lem:convexity}
Suppose that $\kappa >0$, $\kappa > \alpha^2$. Then, there exists $\varepsilon_0^B > 0$ depending on $\kappa$, $\alpha$, $\gamma$, $\beta$ such that for all $\varepsilon \in (0, \varepsilon_0^B)$ the function $(\rho, \eta) \mapsto \f{\tilde{\kappa}}{2}\rho^{2}+\f{\tilde{c}}{2}\eta^{2}+\tilde{\alpha}{\rho\eta}$ is convex. 
\end{lem}
\begin{proof}
We note that the Hessian matrix reads $\begin{pmatrix}
\tilde{\kappa} & \tilde{\alpha}\\
\tilde{\alpha} & \tilde{c}
\end{pmatrix}$ so by the Sylvester's criterion, the desired convexity is equivalent with
$
\tilde{\kappa} > 0, \tilde{\kappa}\, \tilde{c} - \tilde{\alpha}^2 > 0.
$
Concerning the first condition, as $\kappa > 0$ and $\tilde{\kappa}=\f{\kappa}{\eps^{2}}-\gamma$, we can easily find $\varepsilon$ such that $\tilde{\kappa}>0$. Concerning the second condition, we observe that the term standing next to the highest order term $\frac{1}{\varepsilon^4}$ equals $\kappa - \alpha^2$ and it is positive by assumption. Therefore, the conclusion follows.
\end{proof}

\subsection{Construction of weak solutions}

Let $T\in(0,\infty)$ be a final time of existence and consider the following scheme: given a time step size $\tau>0$ and an initial condition $u_{0}=(\rho_{0},\eta_{0})\in \Pro^{2}$ with $E[u_{0}]<+\infty$ we define by induction
\begin{equation}\label{eq:JKO_scheme}
u_{\tau}^{0}:=u_{0},\quad u_{\tau}^{n+1}:=\argmin_{u\in\Pro^{2}}\left\{\f{1}{2\tau}\W_{2}^{2}(u_{\tau}^{n},u)+E_{\eps}[u]\right\},    
\end{equation}
and we set $E_{\varepsilon}[u]=+\infty$ if $u\notin L^{2}(\Td)^{2}.$
\begin{lem}[Existence of minimizers]\label{lem:existence_minimizers}
Let $\tau>0$ and $u_{0}=(\rho_{0},\eta_{0})\in \mathscr{P}_{2}(\Td)^{2}$ with $E_{\eps}[u_{0}]<+\infty$. Then the scheme defined by~\eqref{eq:JKO_scheme} is well-defined. Moreover, we have the following energy estimate
\begin{equation}\label{eq:energy_jko2}
E_{\eps}[u_{\tau}^{N}]+\f{1}{2\tau}\sum_{n=1}^{N}\W^{2}_{2}(u_{\tau}^{n},u_{\tau}^{n-1})\le E_{\eps}[u_0]. 
\end{equation}
\end{lem}

\begin{proof}
\underline{\textit{Step 1: The infimum is bounded.}} Provided $u_{\tau}^{n}$ is defined,  we notice that $u_{\tau}^{n}\in\Pro^{2}$ satisfy $\f{1}{2\tau}\W_{2}^{2}(u_{\tau}^{n},u_{\tau}^{n})+E_{\eps}[u_{\tau}^{n}]=E_{\eps}[u_{\tau}^{n}]<+\infty.$ Therefore the infimum is bounded from above. Since $E_{\eps}[u]$ is bounded from below by Proposition \ref{est:positive_energy_entropy} and $W_{2}^{2}(u_{\tau}^{n},u)$ is nonnegative for all $u$ we also know that the infimum is bounded from below.   

\underline{\textit{Step 2: Candidate for a minimizer.}} We suppose that $u_{\tau}^{n}$ is defined and we want to define $u_{\tau}^{n+1}$. Let $\{u^{k}\}_{k}$ be a minimizing sequence in $\Pro^{2}$ for the problem~\eqref{eq:JKO_scheme}. Without loss of generality, we can assume that for $k$ large enough we have
\begin{equation}\label{eq:est_minimizer_1}
\f{1}{2\tau}\W_{2}^{2}(u_{\tau}^{n},u^{k})+E_{\eps}[u^{k}]\le 2\,E_{\eps}[u_{\tau}^{n}], 
\end{equation}
since the infimum is bounded by $E_{\eps}[u_{\tau}^{n}]$. In particular 
\begin{equation*}
E_{\eps}[u^{k}]\le C(u_{\tau}^{n})    
\end{equation*}
with a constant $C$ independent of $k$. By definition of the energy in~\eqref{eq:energy_weak}, this proves that we can extract from $(u^{k})_{k}$ a subsequence (still denoted by $k$) which converges weakly to some $u$ in $L^{2}(\Td)^{2}$. Now we consider the function $f(\rho,\eta)=\f{\tilde{\kappa}}{2}\rho^{2}+\f{\tilde{c}}{2}\eta^{2}+\tilde{\alpha}{\rho\eta}$ which is convex by Lemma \ref{lem:convexity}. It follows by Tonelli theorem that the functional defined by the first integral of~\eqref{eq:energy_weak} is lower semi-continuous with respect to the $L^{2}$ weak convergence. Since the other terms of~\eqref{eq:energy_weak} are defined with convolutions, these terms are continuous with respect to the $L^{2}$ weak convergence. In the end, $u$ minimizes~\eqref{eq:JKO_scheme}. Finally, $u\in\Pro^{2}$ as a direct consequence of the weak convergence.  

\underline{\textit{Step 3: The energy estimate~\eqref{eq:energy_jko2}.}} Estimate~\eqref{eq:energy_jko2} is a consequence of an induction of the inequality
\begin{equation*}
E_{\eps}[u_{\tau}^{n+1}]+\f{1}{2\tau}\W^{2}_{2}(u_{\tau}^{n+1},u_{\tau}^{n})\le E_{\eps}[u_{\tau}^{n}], 
\end{equation*}
by definition of $u_{\tau}^{n}$. 
\end{proof}

We have constructed a discrete in time sequence. We want to prove that a time-interpolation of this sequence converges to a solution of~\eqref{eq:CHNLW1}-\eqref{eq:CHNLW2}. Let $T>0$ be fixed and $n=\left[\f{T}{\tau}\right]$. We define the interpolation $u_{\tau}(t)=(\rho_{\tau}(t),\eta_{\tau}(t))$ by 
\begin{equation*}
\rho_{\tau}(t)=\rho_{\tau}^{n}, \quad \eta_{\tau}(t)=\eta_{\tau}^{n},\quad t\in((n-1)\tau, n\tau],    
\end{equation*}
where $(\rho_{\tau}^{n},\eta_{\tau}^{n})=u_{\tau}^{n}$ defined in~\eqref{eq:JKO_scheme}. We prove that this sequence is compact in the following lemma:

\begin{lem}[Compactness of the time interpolation sequence]\label{lem:compactness_time}
The sequence of curves $\{u_\tau\}_{\tau}$ is uniformly bounded in $L^{\infty}(0,T; L^2(\Td))$. Moreover, for all $T>0$, there exists an absolutely continuous curve $u:[0,T]\to\Pro^{2}$ such that up to a subsequence, $u_{\tau}(t,\cdot)$ converges to $u(t,\cdot)$ weakly in $L^{2}(\Td)$ as $\tau\to 0$ for all $t \in [0,T]$ and $u_{\tau}$ converges to $u$ in $C([0,T],\Pro^{2})$ as $\tau\to 0$ for all $T>0$. The curve $u$ is globally $1/2$-Hölder continuous in time
\begin{equation}\label{Holder1/2}
\W_{2}(u_{t},u_{s})\le \sqrt{2E_{\eps}[u_{0}]}\sqrt{|s-t|}, \end{equation}
and we have the estimate 
\begin{equation*}
\norm{u}_{L^{\infty}(0,T;L^{2}(\Td))}\le C \,(1 + E_{\eps}[u_{0}]).
\end{equation*}
\end{lem}

\begin{proof}
From~\eqref{eq:energy_jko2} and by definition of $u_{\tau}$ which takes discrete values in time we have that 
\begin{equation}\label{eq:prop_minimizers_summed}
E_{\eps}[u_{\tau}(t)]+\f{1}{2}\int_{0}^{t-\tau}\left(\f{\W_{2}(u_{\tau}(s+\tau),u_{\tau}(s))}{\tau}\right)^{2}\diff s\le E_{\eps}[u_0].    
\end{equation}

With Proposition~\ref{est:positive_energy_entropy}, we obtain
\begin{equation}\label{eq:energy_term_approx_curve}
\int_{\Td}\int_{\Td}\frac{\omega_{\eps}(y)}{\eps^{2}}|u_{\tau}(x)-u_{\tau}(x-y)|^{2}\diff x\diff y\le C(1 + E_{\eps}[u_0])
\end{equation}
so that by Lemma \ref{eq:Poincare_f} we deduce uniform estimate in $L^{\infty}(0,T; L^2(\Omega))$. To prove weak compactness for all times $t\in [0,T]$ and Hölder continuity in time for the limiting curve, it is sufficient to prove
$$
\limsup_{\tau \to 0} \W_{2}(u_{\tau}(s),u_{\tau}(t)) \leq C\, \sqrt{|t-s|}
$$
for some constant $C$ and apply ~\cite[Proposition 3.3.1]{MR2401600} (with $d$ being $\mathcal{W}_2$ distance, $\sigma$ being weak topology on $L^2(\Omega)$, $K$ being the ball in $L^2(\Td)$ such that $u_{\tau}(s) \in K$ for all $s \in [0,T]$, $\mathcal{C} = 
\emptyset$ and $\omega(t,s) = \sqrt{t-s}$). To this end, we write for $0\le s<t$ such that $s\in((m-1)\tau,m\tau]$ and $t\in((n-1)\tau,n\tau]$
\begin{align*}
\W_{2}(u_{\tau}(s),u_{\tau}(t))&\le \sum_{i=m}^{n-1}\W_{2}(u_{\tau}^{i},u_{\tau}^{i+1})\le\left(\sum_{i=m}^{n-1}\W_{2}^{2}(u_{\tau}^{i},u_{\tau}^{i+1})\right)^{1/2}|n-m|^{1/2}\\
&\le \sqrt{2E_{\eps}[u_{0}]}\sqrt{|t-s|+\tau},
\end{align*}
where in the last line we used~\eqref{eq:energy_jko2} and $|n-m|<\f{|t-s|}{\tau}+1$. This concludes the proof.
\end{proof}

\subsection{$H^1$ estimates for the JKO scheme via flow interchange lemma}
The weak convergence of the JKO scheme is not enough to pass to the limit in the definition of weak solutions. We need to obtain better estimates on the solutions. In the Cahn-Hilliard equation, better estimates are derived from considering the functional $\mathcal{U}[\rho,\eta]=\int \rho\log\rho+\eta\log\eta$. This functional generates the heat flow with respect to the Wasserstein distance $\W_{2}$. To improve the regularity we, therefore, use the flow interchange lemma which states that \textit{the dissipation of one functional along the gradient flow of another functional equals the dissipation of the second functional along the gradient flow of the first one}.  

The main result of this section reads:
\begin{proposition}\label{prop:conv_subsequence}
Each solution of the JKO scheme~\eqref{eq:JKO_scheme} satisfies
\begin{equation*}
u_{\tau}^{n}\in H^{1}(\Td)\quad \text{for all $n\in\mathbb{N}$, $\tau>0$} 
\end{equation*}
and the time-interpolation $u_{\tau}=(\rho_{\tau},\eta_{\tau})$ satisfies
\begin{equation*}
\int_{0}^{T}\int_{\Td}\int_{\Td}\f{\omega_{\eps}(y)}{\eps^{2}} (|\nabla\rho_{\tau}(t,x)-\nabla\rho_{\tau}(t,x-y)|^{2}+|\nabla\eta_{\tau}(t,x)-\nabla\eta_{\tau}(t,x-y)|^{2})\diff x\diff y\diff t\le CT,   
\end{equation*}
for all $T>0$. 
Moreover, for every sequence $\tau_{k}\downarrow 0$, we can extract a subsequence (still denoted by $\tau_{k}$) such that for all $T>0$,
\begin{align*}
u_{\tau_{k}}&\to u \text{ strongly in $L^{2}((0,T)\times\Td)$}\\
u_{\tau_{k}}&\rightharpoonup u \text{ weakly in $L^{2}(0,T;H^{1}(\Td))$}.
\end{align*}
\end{proposition}

The main tool to prove Proposition \ref{prop:conv_subsequence} will be the following lemma.

\begin{lem}\label{lem:heatflow}
Let $v_{0}=(\rho_{0},\eta_{0})\in L^{2}(\Td)^{2}$ with $E_{\eps}[v_{0}]<+\infty$. Let $v=(\rho,\eta):[0,+\infty)\to L^{2}(\Td)^{2}$ be a solution of the heat flow
\begin{equation}\label{eq:heatflow}
\begin{cases}
\p_{t}\rho_{t}=\Delta\rho_{t}, \quad \p_{t}\eta_{t}=\Delta\eta_{t}, \quad\text{in $(0,+\infty)\times \Td$}\\
(\rho(0),\eta(0))=(\rho_0,\eta_0).
\end{cases}    
\end{equation}
Suppose that 
\begin{equation}\label{est:crit_heatflow}
\liminf_{s\downarrow 0}\f{1}{s}(E_{\eps}[v_{s}]-E_{\eps}[v_{0}])>-\infty.    
\end{equation}
Then, $v_0\in H^{1}(\Td)$ and for some constant $C$ depending only on the parameters $\kappa$, $\alpha$, $\gamma$, $\beta$ we have
\begin{equation}\label{est:heatflow}
\begin{split}
\int_{\Td}\int_{\Td}\f{\omega_{\eps}(y)}{\eps^{2}} (|\nabla\rho_{0}(x)-\nabla\rho_{0}(x-y)|^{2}+|\nabla\eta_{0}(x)-\nabla\eta_{0}&(x-y)|^{2})\diff x\diff y \leq \\ &\leq -C\, \liminf_{s\downarrow 0}\f{1}{s}(E_{\varepsilon}[v_{s}]-E_{\varepsilon}[v_{0}]).
\end{split}
\end{equation}
\end{lem}

The plan is to initiate the heat flow at the solutions to JKO scheme. If \eqref{est:crit_heatflow} is verified, then \eqref{est:heatflow} will provide the desired $H^1$ estimate by Lemma \ref{lem:poincare_with_parameter}. The crucial information here is that dissipation of Cahn-Hilliard energy along heat flow is related to the dissipation of Cahn-Hilliard entropy. The technical assumption \eqref{est:crit_heatflow} will be verified with the flow interchange lemma which requires a definition of a $\lambda$-flow in $\Pro^{2}$. 

\begin{Def}
Let $\F:\Pro^{2}\to (-\infty,+\infty]$ be a proper lower semi-continuous functional and $\lambda\in\R$. A continuous semigroup $\bS^{t}:\Dom(\F)\to\Dom(\F)$, $t\ge 0$ is a $\lambda$-flow for $\F$ if it satisfies the Evolution Variational Inequality (EVI)
\begin{equation}\label{eq:EVI}
\f{1}{2}\limsup_{h\downarrow 0}\left[\f{\W^{2}_{2}(\bS^{h}u,v)-\W^{2}_{2}(u,v)}{h}\right]+\f{\lambda}{2}\W^{2}_{2}(u,v)+\F(u)\le \F(v),    
\end{equation}
for all measures $u,v\in\Dom(\F)$ with $\W_{2}(u,v)<+\infty$. 
\end{Def}

\begin{lem}[Flow interchange Lemma \cite{MR2581977,MR2921215}]\label{lem.flowinterchange}
  Assume that $\bS_\F$ is a $\lambda$-flow for the proper, lower
  semi-continuous functional $\F$ in $\Pro^{2}$ and
  let $u^n_\tau$ be a $n$-th step approximation constructed by the minimizing movement scheme \eqref{eq:JKO_scheme}.
 If $u^n_\tau\in \Dom(\F)$ then
  \begin{equation}
    \label{eq.flowinterchange}
    \F[u^n_\tau] - \F[u^{n-1}_\tau]
    \leq \tau \liminf_{h\downarrow0} \left(\frac{E_{\eps}[\bS_\F^h(u^n_\tau)]- E_{\eps}[u^n_\tau]}{h}\right)
    - \frac\lambda2 \W^{2}_{2}\big(u^n_\tau,u^{n-1}_\tau\big).
  \end{equation}
\end{lem}
\begin{rem}
In \cite[Lemma 3.2]{MR2581977} Lemma \ref{lem.flowinterchange} was formulated for the case of a bounded domain. However, it also works in the case of torus as it is just a combination of scheme \eqref{eq:JKO_scheme} and the definition of $\lambda$-flow \eqref{eq:EVI}. 
\end{rem}

We are concerned with the functional
\begin{equation*}
\mathcal{U}[\rho,\eta]=\int_{\Td}\rho(x)\log\rho(x)+\eta(x)\log\eta(x)\diff x,    
\end{equation*}
which generates a gradient flow in the product space $(\Pro^{2},\W_{2})$ defined by \eqref{eq:heatflow}. We admit the fairly classical result that the entropy $\mathcal{U}$ possesses a 0-flow given by the heat semigroup.

\begin{proof}[Proof of Lemma \ref{lem:heatflow}]
We recall that $E_{\eps}$ can be written as in~\eqref{eq:intro_energy}
\begin{equation*}
\begin{split}
 E_{\varepsilon}[\rho,\eta]=&\f{1}{4\eps^{2}}\int_{\Td}\int_{\Td}\omega_{\eps}(y)(\kappa|\rho(x)-\rho(x-y)|^{2}+|\eta(x)-\eta(x-y)|^{2})\diff x\diff y  \\ &+\f{\alpha}{2\eps^{2}}\int_{\Td}\int_{\Td}\omega_{\eps}(y)(\rho(x)-\rho(x-y))(\eta(x)-\eta(x-y))\diff x\diff y \\ &-\int_{\Td}\f{\gamma}{2}\rho^{2}+\f{1}{2}\eta^{2}+\beta\rho\eta\diff x.
 \end{split}
 \end{equation*}
From parabolic theory, we know that the solution of~\eqref{eq:heatflow} is smooth for $t>0$ and thus we have after integration by parts
\begin{equation*}
\begin{split}
\f{d}{dt}E_{\eps}[\rho(t,\cdot),v(t,\cdot)]=&-\f{1}{2\eps^{2}}\int_{\Td}\int_{\Td}\omega_{\eps}(y)(\kappa|\nabla\rho(x)-\nabla\rho(x-y)|^{2}+|\nabla\eta(x)-\nabla\eta(x-y)|^{2})\diff x\diff y  \\ &-\f{\alpha}{\eps^{2}}\int_{\Td}\int_{\Td}\omega_{\eps}(y)(\nabla\rho(x)-\nabla\rho(x-y))\cdot(\nabla\eta(x)-\nabla\eta(x-y))\diff x\diff y \\ &+ \int_{\Td}\gamma|\nabla\rho|^{2}+|\nabla\eta|^{2}+2\beta\nabla\rho\cdot\nabla\eta\diff x.
\end{split}
\end{equation*}

Note that we recognize the dissipation of the entropy~\eqref{eq:intro_entropy}. Since the map $t\to v(t,\cdot)$ is continuous in $L^{2}(\Td)^{2}$, we get that the map $t\to E_{\eps}[v(t,\cdot)]$ is continous at $t=0$. By~\eqref{est:crit_heatflow} and Proposition~\ref{est:positive_energy_entropy} there exists $C$ such that for all $t\le t_{0}$ for $t_{0}$ sufficiently small,
\begin{equation}\label{eq:estimate_energy_along_heat_flow}
\begin{split}
&-C \leq \f{E_{\eps}[v(t,\cdot)]-E_{\eps}[v_0]}{t} = \f{1}{t} \int_0^1 \f{d}{ds} E_{\eps}[v(s\,t,\cdot)] \diff s 
\leq \\ 
&\leq -\f{C}{2\eps^{2}}\int_0^1 \int_{\Td}\int_{\Td}\omega_{\eps}(y)(|\nabla\rho(st,x)-\nabla\rho(st,x-y)|^{2}+|\nabla\eta(st,x)-\nabla\eta(st,x-y)|^{2})\diff x\diff y \diff s.
\end{split}
\end{equation}
Therefore, using Lemma~\ref{lem:poincare_with_parameter}, the family $\{\nabla v(t\cdot,\cdot)\}_{t\le t_{0}}$ is bounded in $L^{2}((0,1)\times \Td)$. Now, as $v_0 \in L^2(\Td)$ and $v$ is the solution of heat equation with initial condition $v_0$, $v \in C([0,T]; L^2(\Td))$. Therefore, there exists a bounded modulus of continuity $\pi:[0,T]\to \R^+$ such that $\lim_{t\to 0} \pi(t) = 0$ and
$$
\| v(t,\cdot) - v_0 \|_{L^2(\Td)} \leq \pi(t).
$$
In particular,
$$
\| v(st,\cdot) - v_0 \|_{L^2(\Td)} \leq \pi(st)
$$
so that dominated convergence theorem implies that $v(t\cdot, \cdot) \to v_0$ in $L^2((0,1)\times\Td)$ when $t \to 0$.
Finally, choosing a subsequence such that $\nabla v(t\cdot, \cdot) \weak \xi$ in $L^2((0,1)\times\Td)$, using indentity
$$
\int_0^1\int_{\Td} v(ts,x)\, \mbox{div} \varphi(s,x) \diff s \diff x = - \int_0^1\int_{\Td} \nabla v(ts,x)\, \varphi(s,x) \diff s \diff x
$$
for a smooth test function $\varphi$ and passing to the limit $t \to 0$, we obtain that $\nabla v_0 = \xi$ and so, $v_0 \in H^1(\Td)$. \\

To obtain uniform estimate on $v_0$ in $H^1(\Td)$ we want to pass to the limit $t\to 0$ in \eqref{eq:estimate_energy_along_heat_flow}. For this, we observe that
\begin{multline*}
\int_0^1\int_{\Td}\int_{\Td}\omega_{\eps}(y)|\nabla\rho(st,x)-\nabla\rho(st,x-y)|^{2}\diff x\diff y \diff s= \\ 
=2\int_0^1\int_{\Td}|\nabla\rho(st,x)|^{2}\diff x \diff s-2\int_0^1\int_{\Td}\nabla\rho(st,x)\cdot\nabla(\omega_{\eps}\ast\rho(st,x))\diff x \diff s.
\end{multline*}
The first term is lower semi-continuous with respect to the $L^{2}$ weak topology while the second converges to 
$$
\int_0^1\int_{\Td}\nabla\rho(st,x)\cdot\nabla(\omega_{\eps}\ast\rho(st,x))\diff x \diff s \to \int_{\Td}\nabla\rho_0(x)\cdot\nabla(\omega_{\eps}\ast\rho_0(x))\diff x
$$
as a product of a weakly and strongly convergent sequences. 

\end{proof}

With this lemma, we are finally able to prove 

\begin{proof}[Proof of Proposition \ref{prop:conv_subsequence}]
First, we want to apply the flow interchange lemma to the functional $\mathcal{U}$. This is possible since $\mathcal{U}$ is bounded from below on $\Pro$. Applying Lemma~\ref{lem.flowinterchange} with $\F=\mathcal{U}$ and Lemma~\ref{lem:heatflow} with $v_0=u_{\tau}^{n}$
\begin{equation*}
C\tau\int_{\Td}\int_{\Td}\f{\omega_{\eps}(y)}{\eps^{2}}(|\nabla\rho_{\tau}^{n}(x)-\nabla\rho_{\tau}^n(x-y)|^{2}+|\nabla\eta_{\tau}^{n}(x)-\nabla\eta_{\tau}^{n}(x-y)|^{2})\diff x\diff y\le \mathcal{U}[u_{\tau}^{n-1}]-\mathcal{U}[u_{\tau}^{n}].
\end{equation*}
Summing from $n=1$ to $n=N$ we obtain 
\begin{equation*}
\int_{0}^{T}\int_{\Td}\int_{\Td}\f{\omega_{\eps}(y)}{\eps^2}(|\nabla\rho_{\tau}(t,x)-\nabla\rho_{\tau}(t,x-y)|^{2}+|\nabla\eta_{\tau}(t,x)-\nabla\eta_{\tau}(t,x-y)|^{2})\diff x\diff y\diff t\le CT.   
\end{equation*}
From Lemma~\ref{lem:poincare_with_parameter} we obtain uniform estimate in $L^{2}(0,T;H^{1}(\Td))$ and so, weak compactness in this space.\\

It remains to prove the strong compactness in $L^{2}(0,T;L^{2}(\Td))$. First, we want to apply Theorem \ref{thm:general_lions_aubin} with Banach space $X = L^2(\Td) \times L^2(\Td)$, set $U = \{ u_{\tau}\}_{\tau > 0}$, pseudometric 
$
g(u_1,u_2)= \mathcal{W}_2^2(u_1,u_2)$ (extended to $+\infty$ in case $u_1$ or $u_2$ are not probability measures) and functional $\mathcal{F}$ defined as 
$$
\mathcal{F}(u) = \begin{cases}
\|u\|^2_{H^1(\Td)} &\mbox{ if } u \in H^1(\Td) \times H^1(\Td) \cap \mathcal{P}(\Td) \times \mathcal{P}(\Td), \\
+ \infty &\mbox{ if } u \notin H^1(\Td) \times H^1(\Td) \cap \mathcal{P}(\Td) \times \mathcal{P}(\Td).
\end{cases}
$$
We can do this as $\mathcal{F}$ is lower semicontinuous and its level sets are compact in $L^2(\Td)$ by Reillich-Kondrachov theorem. Furthermore, $g$ is lower semicontinuous \cite[Chapter 6]{MR2459454}. Finally, \eqref{eq:gen_lions_aubin_ass} follows from uniform estimates in $L^2(0,T; H^1(\Td))$ and estimate \eqref{eq:prop_minimizers_summed}.  \\

Therefore, Theorem \ref{eq:gen_lions_aubin_ass} gives us a subsequence (not relabelled) such that
$$
\| u_n(t,\cdot) - u(t,\cdot) \|_{L^2(\Td)}^2 \to 0 \mbox{ for a.e. } t\in [0,T].
$$
As sequence $\{u_n\}$ is bounded in $L^{\infty}(0,T; L^2(\Td))$, the conclusion follows by dominated convergence theorem.
\noindent 
\end{proof}

\subsection{Weak formulation}

To derive the weak formulation, we follow \cite[Section 8.4.2]{MR1964483}. Since~\eqref{eq:JKO_scheme} was derived from the Lagrangian point of view, the idea is to investigate its first variation and prove that it is a time-discrete scheme of~\eqref{eq:CHNLW1}. For this, we introduce a suitable perturbation of $u_{\tau}^{n+1}=(\rho_{\tau}^{n+1},\eta_{\tau}^{n+1})$ as follows: let $\xi$ be a smooth periodic vector field and $T_{\delta}=\Id+\delta\xi$. It is classical to prove that for $\delta$ small enough, $T_\delta$ is a $C^{1}$ diffeomorphism and $\det(\nabla T_\delta)>0$.  Then, we define
\begin{equation*}
\tilde{u}_{\delta}=(T_{\delta}\#\rho_{\tau}^{n+1},\eta_{\tau}^{n+1}). 
\end{equation*}
Note carefully that only the first component was perturbed. This will result in the first equation in Definition~\ref{def:weak_sol_2}. To obtain the second one, it is sufficient to introduce a similar perturbation on the second component of $u_{\tau}^{n+1}$. This results in analogous computations as outlined below and we do not repeat them.\\

Using standard properties of push-forward measure we obtain from \eqref{eq:energy_weak}
\begin{equation}\label{eq:energy_pert}
\begin{split}
E_{\eps}[\tilde{u}_{\delta}]=&
\int_{\Td}\f{\tilde{\kappa}}{2}\f{(\rho_{\tau}^{n+1}(x))^{2}}{\det(\Id+\delta\nabla\xi)}+\f{\tilde{c}}{2}(\eta_{\tau}^{n+1}(x))^{2}+\tilde{\alpha}\rho_{\tau}^{n+1}(x)\,\eta(x+\delta\xi)\diff x\\
&-\f{\tilde{\kappa}+\gamma}{2}\int_{\Td}\int_{\Td}\omega_{\eps}(x-y+\delta(\xi(x)-\xi(y)))\rho_{\tau}^{n+1}(y)\rho_{\tau}^{n+1}(x)\diff y\diff x\\
&-(\tilde{\alpha}+\beta)\int_{\Td}\int_{\Td}\omega_{\eps}(x+\delta\xi(x)-y)\eta_{\tau}^{n+1}(y)\rho_{\tau}^{n+1}(x)\diff y\diff x\\
&-\f{\tilde{c}+1}{2}\int_{\Td}\omega_{\eps}\ast\eta_{\tau}^{n+1}(x) \eta_{\tau}^{n+1}(x) \diff x.
\end{split}
\end{equation}

Using the minimizing property of $u_{\tau}^{n+1}$ in~\eqref{eq:JKO_scheme} gives
\begin{equation}\label{eq_JKO_ineq_with_perturbation}
0\le\f{1}{2\tau}[\W^{2}_{2}(u_{\tau}^{n},\tilde{u}_{\delta})-\W_{2}^{2}(u_{\tau}^{n},u_{\tau}^{n+1})]+ E_{\eps}[\tilde{u}_{\delta}]-E_{\eps}[u_{\tau}^{n+1}].     
\end{equation}

The plan is to expand this inequality in terms of $\delta$, send $\delta \to 0$ and then $\tau \to 0$ which will provide the weak formulation in Definition~\ref{def:weak_sol_2}. We consider three types of terms separately.

\underline{\textit{Step 1: the nonlocal terms in $E_{\eps}[\tilde{u}_{\delta}]-E_{\eps}[u_{\tau}^{n+1}]$.}}
When taking the difference $E_{\eps}[\tilde{u}_{\delta}]-E_{\eps}[u_{\tau}^{n+1}]$, there are two types of nonlocal terms. The first one reads 
\begin{equation*}
-\f{\tilde{\kappa}+\gamma}{2}\int_{\Td}\int_{\Td}[\omega_{\eps}(x-y+\delta(\xi(x)-\xi(y)))-\omega_{\eps}(x-y)]\rho_{\tau}^{n+1}(y)\rho_{\tau}^{n+1}(x)\diff y\diff x.
\end{equation*}
Similarly to~\cite[Proof of Theorem 3.3]{MR3761096} we perform Taylor's expansion and using uniform $L^2$ estimates we obtain that this term is equal to 
\begin{equation*}
-\delta\f{\tilde{\kappa}+\gamma}{2}\int_{\Td}\int_{\Td}\nabla\omega_{\eps}(x-y)\cdot(\xi(x)-\xi(y))\rho_{\tau}^{n+1}(y)\rho_{\tau}^{n+1}(x)\diff y\diff x+o(\delta).    
\end{equation*}

The second term comes from the cross-interaction potentials and we obtain similarly
\begin{equation*}
-\delta(\tilde{\alpha}+\beta)\int_{\Td}\int_{\Td}\nabla\omega_{\eps}(x-y)\cdot\xi(x)\eta_{\tau}^{n+1}(y)\rho_{\tau}^{n+1}(x)\diff y\diff x+o(\delta).    
\end{equation*}

\underline{\textit{Step 2: the local terms.}}
As in Step 1, there are only two differences. The first one reads 
\begin{equation*}
\f{\tilde{\kappa}}{2}\int_{\Td}\f{(\rho_{\tau}^{n+1})^{2}}{\det(\Id+\delta\nabla\xi)}-(\rho_{\tau}^{n+1})^{2}\diff x. 
\end{equation*}
Using $\det(\Id+\delta\nabla\xi)=1+\delta\DIV\xi(x)+o(\delta)$, we obtain that this term is equal to
\begin{equation*}
-\delta\f{\tilde{\kappa}}{2}\int_{\Td}(\rho_{\tau}^{n+1})^{2}\DIV\xi(x)\diff x+o(\delta). 
\end{equation*}
The second term reads 
\begin{equation*}
\int_{\Td}\tilde{\alpha}\rho_{\tau}^{n+1}(x)(\eta_{\tau}^{n+1}(x+\delta\xi)-\eta_{\tau}^{n+1}(x))\diff x.    
\end{equation*}
As $\nabla\eta_{\tau}^{n+1} \in L^{2}(\Td)$, the second term is equal to
\begin{equation*}
\delta\tilde{\alpha}\int_{\Td}\rho_{\tau}^{n+1}(x)\nabla\eta_{\tau}^{n+1}(x)\cdot\xi\diff x + o(\delta).     
\end{equation*}

\underline{\textit{Step 3: The Wasserstein terms}}. 
Since $\rho_{\tau}^{n}$ and $\rho_{\tau}^{n+1}$ are absolutely continuous measures, we know that there exists an optimal map $\nabla\varphi$ such that $\nabla\varphi\#\rho_{\tau}^{n}=\rho_{\tau}^{n+1}$ \cite[Theorem 1.25]{MR3409718} and 
\begin{equation}\label{eq:distance_JKO_Brenier}
\W_{2}^{2}(\rho_{\tau}^{n},\rho_{\tau}^{n+1})=\int_{\Td}\rho_{\tau}^{n}(x)|x-\nabla\varphi(x)+k(x)|^{2}\diff x.     
\end{equation}
where $k(x)\in\mathbb{Z}^{d}$. 
Moreover, we have $T_{\delta}\#\rho_{\tau}^{n+1}=[(\Id+\delta\xi)\circ \nabla\varphi]\#\rho_{\tau}^{n}$. Therefore by definition of the Wasserstein distance, and for $\delta$ small enough we have
\begin{equation}\label{eq:distance_JKO_not_opt}
\W_{2}^{2}(\rho_{\tau}^{n},T_{\delta}\#\rho_{\tau}^{n+1})\le\int_{\Td}\rho_{\tau}^{n}(x)\left|x-\nabla\varphi(x)-\delta\xi\circ\nabla\varphi(x)+k(x)\right|^{2}\diff x.     
\end{equation} 
Using \eqref{eq:distance_JKO_Brenier}--\eqref{eq:distance_JKO_not_opt} and performing Taylor's expansion we obtain
\begin{equation*}
\f{1}{2\tau}[\W^{2}_{2}(u_{\tau}^{n},\tilde{u}_{\delta})-\W_{2}^{2}(u_{\tau}^{n},u_{\tau}^{n+1})]\le \f{\delta}{\tau}\int_{\Td}\rho_{\tau}^{n}(x)(\nabla\varphi(x)-x+k(x))\cdot(\xi\circ\nabla\varphi(x))\diff x + o(\delta). 
\end{equation*}

\underline{\textit{Weak formulation.}} We plug inequalities from Steps 1-3 to \eqref{eq_JKO_ineq_with_perturbation}. As $\xi$ can be replaced with $-\xi$ we obtain equality
\begin{equation}\label{eq:weak_form_existence_first}
\begin{split}
\f{1}{\tau}\int_{\Td}\rho_{\tau}^{n}(x)(\nabla\varphi(x)&-x+k(x))\cdot(\xi\circ\nabla\varphi(x))\diff x=\\
&=\f{\tilde{\kappa}}{2}\int_{\Td}(\rho_{\tau}^{n+1}(x))^{2}\DIV\xi(x)\diff x-\tilde{\alpha}\int_{\Td}\rho_{\tau}^{n+1}(x)\nabla\eta_{\tau}^{n+1}(x)\cdot\xi\diff x \\ 
&+(\tilde{\alpha}+\beta)\int_{\Td}\int_{\Td}\nabla\omega_{\eps}(x-y)\cdot\xi(x)\eta_{\tau}^{n+1}(y)\rho_{\tau}^{n+1}(x)\diff y\diff x\\
&+\f{\tilde{\kappa}+\gamma}{2}\int_{\Td}\int_{\Td}\nabla\omega_{\eps}(x-y)\cdot(\xi(x)-\xi(y))\rho_{\tau}^{n+1}(y)\rho_{\tau}^{n+1}(x)\diff y\diff x.
\end{split}    
\end{equation}
Now we consider test function $\xi=\nabla\zeta$ for some $\zeta\in C^{\infty}(\Td)$. By periodicity we have $\zeta(\nabla\varphi(x))-\zeta(x)=\zeta(\nabla\varphi(x))-\zeta(x-k(x))$. Therefore, we have
\begin{equation*}
\zeta(\nabla\varphi(x))-\zeta(x)=(\nabla\varphi(x)-x+k(x))\cdot(\nabla\zeta\circ\nabla\varphi(x))+ O(|x-k(x)-\nabla\varphi(x)|^{2})    
\end{equation*}
and $\rho_{\tau}^{n+1} = \nabla\varphi \# \rho_{\tau}^{n}$, we obtain that the (LHS) of \eqref{eq:weak_form_existence_first} is equal to 
\begin{equation*}
\f{1}{\tau}\left(\int_{\Td}\rho_{\tau}^{n+1}\zeta-\int_{\Td}\rho_{\tau}^{n}\zeta\right)+O\left(\f{\W_{2}^{2}(\rho_{\tau}^{n},\rho_{\tau}^{n+1})}{\tau}\right).
\end{equation*}
Now, let $t_1$, $t_2$ be arbitrary. As the curve $\rho_{\tau}$ is piecewisely constant, we can sum up from $n=n_{1}=[t_{1}/\tau]$ to $n_{2}=[t_{2}/\tau]+1$ and obtain
\begin{equation*}
\begin{split}
\int_{\Td}\rho_{\tau}(t_{2})\zeta\diff &x-\int_{\Td}\rho_{\tau}(t_{1})\zeta\diff x+O\left(\sum_{n=n_{1}}^{n_{2}}\W_{2}^{2}(\rho_{\tau}^{n},\rho_{\tau}^{n+1})\right)= \\
&=\f{\tilde{\kappa}}{2}\int_{t_{1}}^{t_{2}}\int_{\Td}\rho_{\tau}(t)^{2}\Delta\zeta(x)\diff x\diff t -\tilde{\alpha}\int_{t_{1}}^{t_{2}}\int_{\Td}\rho_{\tau}(t,x)\nabla\eta_{\tau}(t,x)\cdot\nabla\zeta\diff x\diff t\\
&\phantom{=}+(\tilde{\alpha}+\beta)\int_{t_{1}}^{t_{2}}\int_{\Td}\int_{\Td}\nabla\omega_{\eps}(x-y)\cdot\nabla\zeta(x)\eta_{\tau}(t,y)\rho_{\tau}(t,x)\diff y\diff x\diff t\\
&\phantom{=}+\f{\tilde{\kappa}+\gamma}{2}\int_{t_{1}}^{t_{2}}\int_{\Td}\int_{\Td}\nabla\omega_{\eps}(x-y)\cdot(\nabla\zeta(x)-\nabla\zeta(y))\rho_{\tau}(t,y)\rho_{\tau}(t,x)\diff y\diff x\diff t+ O(\tau),
\end{split}    
\end{equation*}
where the term $O(\tau)$ appears because each term has at least one term ($\rho_{\tau}$ or $\eta_{\tau}$) which is bounded in $L^{\infty}(0,T; L^2(\Td))$. Using the energy estimate~\eqref{eq:energy_jko2} on the Wasserstein distance, we can incorporate the term $O\left(\sum_{n=n_{1}}^{n_{2}}\W_{2}^{2}(\rho_{\tau}^{n},\rho_{\tau}^{n+1})\right)$ into $O(\tau)$. Sending $\tau\to 0$, using Proposition~\ref{prop:conv_subsequence} and pointwise (in time) weak convergence (in space) from Lemma \ref{lem:compactness_time} yields
\begin{equation}\label{eq:weak_formulation_almost}
\begin{split}
\int_{\Td}\rho(t_{2},x)\zeta(x)\diff &x-\int_{\Td}\rho(t_{1},x)\zeta(x)\diff x=\\
& = \f{\tilde{\kappa}}{2}\int_{t_{1}}^{t_{2}}\int_{\Td}\rho(t,x)^{2}\Delta\zeta(x)\diff x\diff t -\tilde{\alpha}\int_{t_{1}}^{t_{2}}\int_{\Td}\rho(t,x)\nabla\eta(t,x)\cdot\nabla\zeta\diff x\diff t\\
&\phantom{=} + (\tilde{\alpha}+\beta)\int_{t_{1}}^{t_{2}}\int_{\Td}\int_{\Td}\nabla\omega_{\eps}(x-y)\cdot\nabla\zeta(x)\eta(t,y)\rho(t,x)\diff y\diff x\diff t\\
&\phantom{=} +\f{\tilde{\kappa}+\gamma}{2}\int_{t_{1}}^{t_{2}}\int_{\Td}\int_{\Td}\nabla\omega_{\eps}(x-y)\cdot(\nabla\zeta(x)-\nabla\zeta(y))\rho(t,y)\rho(t,x)\diff y\diff x\diff t.
\end{split}    
\end{equation}
Performing integration by parts we obtain
$$
\f{\tilde{\kappa}}{2}\int_{t_{1}}^{t_{2}}\int_{\Td}\rho(t,x)^{2}\Delta\zeta(x)\diff x\diff t = 
\tilde{\kappa} \int_{t_{1}}^{t_{2}}\int_{\Td} \rho(t,x) \nabla \rho(t,x) \cdot \nabla \zeta(x) \diff x\diff t,
$$
while changing variables we obtain
\begin{multline*}
\f{\tilde{\kappa}+\gamma}{2} \int_{t_{1}}^{t_{2}}\int_{\Td}\int_{\Td}\nabla\omega_{\eps}(x-y)\cdot(\nabla\zeta(x)-\nabla\zeta(y))\rho(t,y)\rho(t,x)\diff y\diff x\diff t = \\ 
=
-(\tilde{\kappa}+\gamma)\, 
\int_{t_{1}}^{t_{2}}\int_{\Td}
\omega_{\eps} \ast \rho (t,x) \cdot \nabla \zeta(x)
\diff x\diff t.
\end{multline*}
Having these two observations in mind, we obtain the weak formulation with test function $\zeta(x)$ depending only on $x$. The general weak formulation with test functions depending on $t$ and $x$ as in Definition \ref{def:weak_sol_2} follows from multiplying with $\partial_t \psi(t)$, integrating in time and using the classical density of functions of the form $\psi(t)\,\zeta(x)$ over the set of test functions $\varphi(t,x)$ (see \cite[Theorem D.5]{MR4309603}).

\underline{Regularity estimates \eqref{eq:thm_regularity_compactness_rho_eta}--\eqref{eq:thm_regularity_compactness_nabla_rho_eta}.} Clearly, it is sufficient to prove these estimates for the first component of $u = (\rho,\eta)$. For \eqref{eq:thm_regularity_compactness_rho_eta} we note that from Lemma \ref{lem:compactness_time} we have that $u_{\tau}(t,\cdot) \weak u(t,\cdot)$ for all fixed $t \in [0,T]$. Moreover, we can write the quantity of interest as
$$
\int_{\Td}\int_{\Td} {\omega_{\eps}(y)}|\rho_\tau(t,x)-\rho_\tau(t,x-y)|^{2}\diff x\diff y = 2 \int_{\Td} |\rho_{\tau}(t,x)|^2 \diff x - 2 \int_{\Td} \rho_{\tau}(t,x)\, \omega_{\eps} \ast \rho_{\tau}(t,x) \diff x
$$
so that the first term is weakly lower semicontinuous while the second is the product of weakly and strongly converging sequences. Applying $\liminf_{\eps \to 0}$ to \eqref{eq:energy_term_approx_curve} we deduce \eqref{eq:thm_regularity_compactness_rho_eta}. Concerning \eqref{eq:thm_regularity_compactness_nabla_rho_eta}, the proof is carried out in the similar way: this time we use strong compactness of $\rho_{\tau}$ and weak compactness of $\nabla \rho_{\tau}$ in $L^2((0,T)\times \Td)$ from Lemma \ref{prop:conv_subsequence} which allow to handle integral with respect to time.

\section{Limit $\eps\to 0$}\label{sec:conv_eps}

We want to send $\eps \to 0$ and obtain convergence of weak solutions of the nonlocal Cahn-Hilliard system~\eqref{eq:CHNL1}-\eqref{eq:CHNL2} to weak solutions of the local version of the Cahn-Hilliard system. We define weak solutions of the latter.

\begin{Def}\label{def:weak_sol_local}
We say that $u=(\rho(\cdot),\eta(\cdot)):[0,\infty)\to \Pro^{2}$ is a weak solution of~\eqref{eq:CHL1}-\eqref{eq:CHL2} with initial condition $(\rho_{0},\eta_{0})$ if $\rho,\eta\in L^{2}(0,T;H^{2}(\Td))$ for all $T>0$, and if for all  $\varphi,\phi\in C_{c}^{\infty}([0,+\infty)\times\Td)$ we have
\begin{equation*}
\begin{split}
&-\int_{0}^{\infty}\int_{\Td}\rho\p_{t}\varphi\diff x\diff t-\int_{\Td}\rho_{0}\varphi(0,x)\diff x= -\kappa\int_{0}^{\infty}\int_{\Td}\Delta\rho\nabla \rho\cdot\nabla\varphi\diff x\diff t\\
&-\kappa\int_{0}^{\infty}\int_{\Td}\rho\Delta\rho\Delta\varphi\diff x\diff t- \alpha\int_{0}^{\infty}\int_{\Td}D^{2}\eta:(\nabla\rho\otimes\nabla\varphi+ \rho D^{2}\varphi) \diff x\diff t -\int_{0}^{\infty}\int_{\Td}\rho(\gamma\nabla\rho-\beta\nabla\eta)\cdot\nabla\varphi\diff x\diff t,
\end{split}    
\end{equation*}
\begin{equation*}
\begin{split}
&-\int_{0}^{\infty}\int_{\Td}\eta\p_{t}\phi\diff x\diff t-\int_{\Td}\eta_{0}\phi(0,x)\diff x=-\int_{0}^{\infty}\int_{\Td}\Delta\eta\nabla \eta\cdot\nabla\phi\diff x\diff t\\
&-\int_{0}^{\infty}\int_{\Td}\eta\Delta\eta\Delta\phi\diff x\diff t- \alpha\int_{0}^{\infty}\int_{\Td}D^{2}\rho:(\nabla\eta\otimes\nabla\phi+ \eta D^{2}\phi) \diff x\diff t -\int_{0}^{\infty}\int_{\Td}\eta(\nabla\eta-\beta\nabla\rho)\cdot\nabla\phi\diff x\diff t.
\end{split}    
\end{equation*}
\end{Def}

As we will see (Lemma \ref{lem:estimates_uniform_eps}), we have bounds at most on the gradient of $\nabla \rho_{\eps},\nabla\eta_{\eps}$, and the limit equation has four derivatives. That means we need to mimic at the epsilon level integration by parts for nonlocal operators. For that purpose, we define the operator 
\begin{equation}\label{eq:nonlocal_gradient}
S_{\eps}[\varphi](x,y):=\f{\sqrt{\omega_{\eps}(y)}}{\sqrt{2}\eps}(\varphi(x-y)-\varphi(x)) 
\end{equation}
which has the following properties, see~\cite[Lemma 3.4]{elbar-skrzeczkowski}:
\begin{lem}\label{lem:S_properties} The operator $S_{\eps}$ satisfies:
\begin{enumerate}[label=(S\arabic*)]
    \item $S_{\eps}$ is a linear operator that commutes with derivatives with respect to $x$,
    \item\label{propS_product_rule} for all functions $f,g: \Td \to \R$ we have
    \begin{align*}S_{\eps}[fg](x,y)-S_{\eps}[f](x,y)g(x)&-S_{\eps}[g](x,y)f(x)= \\ &= \f{\sqrt{\omega_{\eps}(y)}}{\sqrt{2}\eps}[(f(x-y)-f(x))(g(x-y)-g(x))].
    \end{align*}
    \item\label{propS_nonneg} for all $u, \varphi \in L^2(\Td)$
    $$
    \langle B_{\eps}[u](\cdot),\varphi(\cdot)\rangle_{L^{2}(\Td)}=\langle S_{\eps}[u](\cdot,\cdot),S_{\eps}[\varphi](\cdot,\cdot)\rangle_{L^{2}(\Td \times \Td)}.
    $$
    \item\label{propS:conv_to_energy}
    if $\{u_{\varepsilon}\}$ is strongly compact in $L^2(0,T; H^1(\Td))$ and $\varphi \in L^{\infty}((0,T)\times \Td)$ we have
    $$
    \int_0^T \int_{\Td} \int_{\Td} (S_{\varepsilon}[u_{\varepsilon}])^2 \, \varphi(t,x) \to \int_0^T \int_{\Td} |\nabla u(t,x)|^2 \, \varphi(t,x).
    $$
\end{enumerate}
\end{lem}

\subsection{Uniform estimates in $\eps$ and compactness} We collect here estimates for the solutions of \eqref{eq:CHNL1}--\eqref{eq:CHNL2} which are uniform in $\eps>0$.
\begin{lem}\label{lem:estimates_uniform_eps}
Let $(\rho_{\eps},\eta_{\eps})$ be the solution of \eqref{eq:CHNL1}--\eqref{eq:CHNL2} constructed in Theorem \ref{thm:weaksoldelta}. Then, the following sequences are bounded:
\begin{enumerate}
    \item $\{\rho_{\eps}\}$, $\{ \eta_{\eps}\}$ in $L^{\infty}(0,T; L^2(\Omega))$ and $L^2(0,T; H^1(\Omega))$,
    \item $\{\partial_t \rho_{\eps}\}$, $\{\partial_t \eta_{\eps}\}$ in $L^2(0,T; H^{-2 - \frac{d}{2}}(\Omega))$.
\end{enumerate}
Moreover, the sequences $\{\rho_{\eps}\}$, $\{ \eta_{\eps}\}$ are strongly compact in $L^2(0,T; H^1(\Omega))$ and the limits belong to $L^2(0,T; H^2(\Omega))$.
\end{lem}
\begin{rem}
We note that we cannot use estimates coming from dissipation of the energy \eqref{eq:energy_diss} which can provide better estimates on time derivatives because this information is lost in the JKO scheme. We proceed with a different approach.
\end{rem}
\begin{proof}[Proof of Lemma \ref{lem:estimates_uniform_eps}] The first estimate follows from \eqref{eq:thm_regularity_compactness_rho_eta}--\eqref{eq:thm_regularity_compactness_nabla_rho_eta} and Lemma \ref{lem:poincare_with_parameter}. To see the second, we consider equation for $\rho_\eps$ and test it with a smooth and compactly supported function $\varphi(t,x)$. We need to control the following terms
$$
\int_{0}^{T}\int_{\Td} \rho_{\eps}\,\nabla B_{\eps}[\rho_{\eps}] \cdot\nabla\varphi\diff x\diff t,
\quad  \int_{0}^{T}\int_{\Td}\rho_{\eps}\, \nabla B_{\eps}[\eta_{\eps}]\cdot\nabla\varphi \diff x\diff t, \quad \int_{0}^{T}\int_{\Td}\rho_{\eps} (\gamma\nabla\rho_{\eps}-\beta\nabla\eta_{\eps})\cdot\nabla\varphi\diff x\diff t.
$$
The first two terms are controlled in the same way so that we focus on the second. Using \ref{propS_nonneg} in Lemma \ref{lem:S_properties} and $\nabla B_{\varepsilon}[\eta_{\eps}] = B_{\varepsilon}[\nabla \eta_{\eps}]$ we have
$$
 \int_{0}^{T}\int_{\Td}\rho_{\eps}\, \nabla B_{\eps}[\eta_{\eps}]\cdot\nabla\varphi \diff x\diff t =
  \int_{0}^{T}\int_{\Td} \int_{\Td}   S_{\eps}[\nabla \eta_{\eps}]\cdot S_{\eps}[\rho_{\eps}\,\nabla\varphi] \diff y \diff x\diff t
$$
which is bounded if $\varphi \in L^2(0,T; H^2(\Td))$. Concerning the third term, it is controlled when $\nabla \varphi \in L^2(0,T; L^{\infty}(\Td))$ which is implied by $\varphi \in L^2(0,T; H^{2+\frac{d}{2}}(\Td))$ by Sobolev embedding. The conclusion follows.\\

\noindent The strong compactness follows from Lemma \ref{lem:Lions-Aubin_nonlocal} and estimates \eqref{eq:thm_regularity_compactness_rho_eta}--\eqref{eq:thm_regularity_compactness_nabla_rho_eta}.

\end{proof}

\subsection{Convergence $\eps\to 0$}\label{subsec:conv_eps}

\begin{proof}[Proof of Theorem~\ref{thm:final}] As Equations~\eqref{eq:CHNL1} and~\eqref{eq:CHNL2} have a similar structure, we focus only on Equation~\eqref{eq:CHNL1}. More precisely, we pass to the limit in the term $\int_{0}^{\infty}\int_{\Td}\DIV(\rho_{\eps}\nabla\mu_{\rho,\eps})\varphi\diff x\diff t$ where $\varphi \in C^3([0,\infty)\times \Td)$. Integrating by parts, we obtain 
\begin{equation}\label{eq:split_for_I1_I2_I3}
    \begin{split}
\int_{0}^{\infty}\int_{\Td}\DIV(&\rho_{\eps}\nabla\mu_{\rho,\eps})\,\varphi\diff x\diff t = -\int_{0}^{\infty}\int_{\Td}\rho_{\eps}\nabla\mu_{\rho,\eps}\cdot\nabla\varphi\diff x\diff t \\ &=\kappa\int_{0}^{\infty}\int_{\Td}B_{\eps}[\rho_{\eps}]\nabla \rho_{\eps}\cdot\nabla\varphi\diff x\diff t
+\kappa\int_{0}^{\infty}\int_{\Td}B_{\eps}[\rho_{\eps}]\rho_{\eps}\Delta\varphi\diff x\diff t\\ &\phantom{= }-\alpha\int_{0}^{\infty}\int_{\Td}\rho_{\eps}B_{\eps}[\nabla\eta_\eps]\cdot\nabla\varphi \diff x\diff t -\int_{0}^{\infty}\int_{\Td}\rho_{\eps}(\gamma\nabla\rho_{\eps}-\beta\nabla\eta_{\eps})\cdot\nabla\varphi\diff x\diff t\\
&=:I_{1}+I_{2}+I_{3}+I_{4}. \phantom{\int_{0}^{\infty}\int_{\Td}}
\end{split}
\end{equation}
Concerning the term $I_4$, its convergence is straightforward because all the sequences $\{\rho_{\eps}\}$, $\{\eta_{\eps}\}$, $\{\nabla \rho_{\eps}\}$ are compact in $L^2((0,T)\times \Td)$. For the passage to the limit in $I_{1}, I_{2}$ we refer to~\cite[Steps 1, 2; Proof of Theorem 1.8]{elbar-skrzeczkowski} (these are exactly the terms that appear for analysis of a single equation). We now prove the convergence of the term $I_{3}$. Due to Lemma \ref{lem:S_properties} \ref{propS_product_rule}, we have (omitting the constant $\alpha$)
\begin{align*}
-I_3&=\int_{0}^{\infty}\int_{\Td}S_{\eps}[\rho_{\eps}]S_{\eps}[\nabla\eta_\eps]\cdot\nabla\varphi \diff x\diff t+\int_{0}^{\infty}\int_{\Td}\rho_{\eps}S_{\eps}[\nabla\eta_\eps]\cdot S_{\eps}[\nabla\varphi] \diff x\diff t+R_{\eps}=:J_{1}+J_{2}+R_{\eps}, 
\end{align*}
where $R_{\eps}$ is defined as
\begin{equation*}
 R_{\eps}= \int_{0}^{\infty}\int_{\Td}\int_{\Td}S_{\eps}[\nabla\eta_\eps]\cdot\f{\sqrt{w_{\eps}(y)}}{\sqrt{2}\eps}[(\nabla\varphi(x-y)-\nabla\varphi(x))\cdot(\rho_{\eps}(x-y)-\rho_{\eps}(x))]\diff x\diff y \diff t.
\end{equation*}

After a change of variables $y\to\f{y}{\eps}$ and using the definition of $\omega_{\eps}$ we obtain
\begin{align}
J_{1}=\f{1}{2}\int_{0}^{\infty}\int_{\Td}\int_{\Td}\omega(y)\f{\rho_{\eps}(x)-\rho_{\eps}(x-\eps y)}{\eps}\f{\nabla\eta_{\eps}(x)-\nabla\eta_{\eps}(x-\eps y)}{\eps}\cdot\nabla\varphi(x)\diff x\diff y\diff t.
\end{align}
From Lemma \ref{lem:diff_quot_strong_conv} and \ref{lem:weak_conv_diff_quot} we obtain strong convergence of $\f{\rho_{\eps}(x)-\rho_{\eps}(x-\eps y)}{\eps} \to \nabla\rho(x)\cdot y$ and weak convergence of $\sqrt{\omega(y)}\,\f{\nabla\eta_{\eps}(x)-\nabla\eta_{\eps}(x-\eps y)}{\eps} \weak \sqrt{\omega(y)}\,D^2 \eta(x) \cdot y$ (both in $L^2((0,T)\times\Td\times\Td)$) so that we easily conclude
$$
J_{1}\to \f{1}{2}\int_{\Td}\omega(y)\int_{0}^{\infty}\int_{\Td}(\nabla\rho(x)\cdot y)\, (D^{2}\eta(x)y) \cdot\nabla\varphi(x)\diff y\diff x\diff t. 
$$

where $D^{2}\eta$ denotes the Hessian matrix $(\p_{ij} \eta)_{i,j}$. 

For $J_{2}$ we similarly write
\begin{align}
J_{2}=\f{1}{2}\int_{0}^{\infty}\int_{\Td}\int_{\Td}\omega(y) \, \rho_{\eps}(x) \, \f{\nabla\eta_{\eps}(x)-\nabla\eta_{\eps}(x-\eps y)}{\eps}\cdot\f{\nabla\varphi(x) - \nabla\varphi(x-\eps y)}{\eps}\diff x\diff y\diff t.
\end{align}
and the same argument as for $J_1$ shows that 
$$
J_{2}\to \f{1}{2}\int_{\Td}\omega(y)\int_{0}^{\infty}\int_{\Td}\rho(x)\, (D^{2}\eta(x)y) \cdot(D^{2}\varphi(x)y)\diff y\diff x\diff t. 
$$

By properties of $\omega$, we obtain 
\begin{equation*}
J_{1}+J_{2}=\int_{0}^{\infty}\int_{\Td}D^{2}\eta:(\nabla\rho\otimes\nabla\varphi+\rho D^{2}\varphi)\diff x\diff t.    
\end{equation*}

It remains to show $R_{\eps} \to 0$. By Cauchy-Schwarz inequality (in time and space) as well as bounds on $S_{\eps}[\nabla\eta_{\eps}]$ from the entropy it remains to prove that
\begin{equation*}
\int_{0}^{\infty}\int_{\Td} \int_{\Td}\f{\omega_{\eps}(y)}{\eps^2}|\rho_{\eps}(x-y)-\rho_{\eps}(x)|^{2}|\nabla\varphi(x-y)-\nabla\varphi(x)|^{2}\diff y \diff x\diff t\to 0.     
\end{equation*}
Using Taylor's expansion we can estimate this integral with
\begin{equation*}
\eps\norm{D^{2}\varphi}_{L^{\infty}} \, \int_{0}^{\infty}\int_{\Td}\int_{\Td}\f{\omega_{\eps}(y)}{\eps^2}|\rho_{\eps}(x-y)-\rho_{\eps}(x)|^{2} \diff y\diff x\diff t 
\end{equation*}
which converges to zero by the bound from the entropy so that $R_{\eps}\to 0$. 
\end{proof}

\appendix

\section{Results on difference quotients}

\begin{lem}\label{lem:diff_quot_strong_conv}
Let $\{\ueps\}$ be a sequence strongly compact in $L^2(0,T; H^1(\Td))$. Then, for fixed $y \in \Td$,
$$
\frac{u_{\eps}(t,x-\eps y) - u_{\eps}(t,x)}{\varepsilon} \to - \nabla u(t,x) \cdot y \mbox{ strongly in } L^2((0,T)\times \Td \times \Td).
$$
\end{lem}
\begin{proof}
Clearly, the sequence converges in $L^2((0,T)\times \Td)$, that is
$$
\int_0^T \int_{\Td} \left|\frac{u_{\eps}(t,x-\eps y) - u_{\eps}(t,x)}{\varepsilon} - \nabla u(t,x) \cdot y \right|^2 \diff x \diff t \to 0.
$$
To see the convergence in $L^2((0,T)\times \Td \times \Td)$, it is sufficient to apply dominated convergence theorem as we have the estimate
$$
\int_0^T \int_{\Td} \left|\frac{u_{\eps}(t,x-\eps y) - u_{\eps}(t,x)}{\varepsilon} \right|^2 \leq \int_0^T \int_{\Td} \left|\nabla u^{\eps} \cdot y \right|^2 \diff x \diff t \leq |y| \, \sup_{\varepsilon} \|\nabla \ueps\|_{L^2_{t,x}}^2.
$$
\end{proof}

\begin{lem}\label{lem:weak_conv_diff_quot} Let $\varphi \in L^{\infty}(\Td)$ and $\{\eta_{\eps}\}$ be a sequence such that
\begin{itemize}
    \item $\varphi(y)\,\f{\nabla\eta_{\eps}(x)-\nabla\eta_{\eps}(x-\eps y)}{\eps}$ is bounded in $L^2((0,T)\times\Td \times \Td)$,
    \item $\eta_{\eps}(t,x) \weak \eta(t,x)$ in $L^2(0,T; H^1(\Td))$ and $\eta \in L^2(0,T; H^2(\Td))$.
\end{itemize}
Then,
$$
\varphi(y)\,\f{\nabla\eta_{\eps}(x)-\nabla\eta_{\eps}(x-\eps y)}{\eps} \weak \varphi(y)\,D^2 \eta(t,x) \cdot y \mbox{ weakly in } L^2((0,T)\times\Td \times \Td).
$$
\end{lem}
\begin{proof}
Clearly, after passing to a subsequence, the limit exists and we only need to identify it. For this, we consider a smooth and compactly supported test function $\psi(t,x,y)$ and compute
\begin{multline*}
    \int_0^T \int_{\Td} \int_{\Td} 
    \varphi(y)\,\f{\nabla\eta_{\eps}(t,x)-\nabla\eta_{\eps}(t,x-\eps y)}{\eps} \psi(t,x,y)
    \diff y \diff x \diff t= \\ = 
    \int_0^T \int_{\Td} \int_{\Td} 
    \varphi(y)\,\f{\psi(t,x,y)-\psi(t,x+\eps y,y)}{\eps} \nabla \eta_{\eps}(t,x)
    \diff y \diff x \diff t
\end{multline*}
which converges to
$$
-\int_0^T \int_{\Td} \int_{\Td}  \varphi(y) \nabla \psi(t,x,y) \cdot y \, \nabla \eta(t,x) \diff y \diff x \diff t = \int_0^T \int_{\Td} \int_{\Td}\varphi(y) \, \psi(t,x,y) D^2\eta(t,x) \cdot y \diff y \diff x \diff t
$$
because $\eta \in L^2(0,T; H^2(\Td))$.

\end{proof}

\section{Compactness results}
\subsection{A version of Lions-Aubin for JKO scheme.} We recall here from \cite[Theorem 2.1]{MR3761096} a version of Lions-Aubin lemma useful for establishing compactness of a sequence of solutions to JKO scheme. For the proof we refer to \cite[Theorem 2]{MR2005609}.
\begin{thm}\label{thm:general_lions_aubin}
Let $(X, \|\cdot\|_{X})$ be a Banach space. We consider 
\begin{itemize}
    \item a lower semi-continuous functional $\mathcal{F}:X \to [0,+\infty]$ with relatively compact sublevels in $X$,
    \item a pseudo-distance $g:X \times X \to [0, +\infty]$, that is $g$ is lower semicontinuous and $g(\rho,\eta)=0$ for some $\rho, \eta \in X$ such that $\mathcal{F}(\rho), \mathcal{F}(\eta)<\infty$ implies $\rho = \eta$.
\end{itemize}
Let $U$ be a set of measurable functions $u: (0,T) \times X$ with $T>0$ fixed. Assume further that
\begin{equation}\label{eq:gen_lions_aubin_ass}
\sup_{u \in U} \int_0^T \mathcal{F}(u(t))\diff t < \infty, \qquad \lim_{h\to 0} \sup_{u \in U} g(u(t+h), u(t)) \diff t  = 0.
\end{equation}
Then, $U$ contains a sequence $\{u_n\}$ converging in measure to some $u \in X$, i.e.
$$
\forall_{\varepsilon>0} \left| \{t \in [0,T]:  \|u_n - u \|_{X} > \varepsilon \} \right| \to 0 \mbox{ as } n\to \infty.
$$
In particular, there exists a subsequence (not relabelled) such that
$$
u_n(t) \to u(t) \mbox{ in } X \mbox{ for a.e. } t\in [0,T].
$$
\end{thm}

\subsection{Nonlocal version of Lions-Aubin lemma}

The following result was proved in \cite[Theorem B.1]{elbar-skrzeczkowski} based on \cite{bourgain2001another} and  \cite{MR2041005}. In fact, the proof of Theorem \ref{thm:ponce_tx} is the proof of Proposition \ref{prop:ponce} integrated in time.
\begin{thm}\label{thm:ponce_tx}
Let $d \geq 2$. Let $\{f_\eps\}$ be a sequence bounded in $L^p((0,T)\times \Td)$. Suppose that there exists a sequence $\{\rho_{\eps}\}$ as above such that
\begin{equation}\label{eq:Ponce_org_condition}
\int_0^T \int_{\Td} \int_{\Td} \frac{|f_\eps(t,x) - f_\eps(t,y)|^p}{\varepsilon^p} \omega_\eps(|x-y|) \diff x \diff y \diff t \leq C
\end{equation}
for some constant $C$. Then, $\{f_\eps\}$ is compact in space in $L^p((0,T)\times \Td)$, i.e.
\begin{equation}\label{eq:equicontinuity_Lp_space}
\lim_{\delta \to 0} \limsup_{\varepsilon \to 0} \int_0^T \int_{\Td} |f_\eps \ast \varphi_{\delta}(t,x) - f_\eps(t,x)|^p \diff x \diff t = 0
\end{equation}
for all families of mollifiers $\{\varphi_{\delta}\}_{0<\delta<1}$.
\end{thm}
We prove here:
\begin{lem}\label{lem:Lions-Aubin_nonlocal}
Suppose that $\{f_{\varepsilon}\}$ is a sequence bounded in $L^2((0,T)\times \Td)$ such that
\begin{itemize}
    \item $\{\p_{t}f_{\eps}\}$ is uniformly bounded in $L^2(0,T; H^{-k}(\Td))$ for some $k \in \N$,
    \item  $\{f_\eps\}$ is compact in space in $L^2((0,T)\times \Td)$, i.e.
\begin{equation}\label{item2compactness_2}
\lim_{\delta \to 0} \limsup_{\varepsilon \to 0} \int_0^T \int_{\Td} |f_\eps \ast \varphi_{\delta}(t,x) - f_\eps(t,x)|^2 \diff x \diff t = 0
\end{equation}
for some family of mollifiers $\{\varphi_{\delta}\}_{0<\delta<1}$.
\end{itemize}
Then, $\{f_\eps\}$ is compact in time in $L^2((0,T)\times \Td)$, i.e.
\begin{equation}\label{item2compactness}
\lim_{h \to 0} \limsup_{\varepsilon \to 0} \int_0^{T-h} \int_{\Td} |f_\eps(t+h,x) - f_\eps(t,x)|^2 \diff x \diff t \to 0 \mbox{ as } h \to 0
\end{equation}
and so, it is compact in $L^2((0,T)\times \Td)$.
\end{lem}

\begin{proof}
Using the mollifiers with $\delta = \delta(h)$ depending on $h$ to be specified later in the way that $\delta(h) \to 0$ as $h\to 0$, we first split
\begin{align*}
\int_{0}^{T-h}\int_{\Td}|f_{\eps}(t+h,x)-f_{\eps}(t,x)&|^2\diff x \diff t\le 4\int_{0}^{T-h}\int_{\Td}|f_{\eps}(t,x)-f_{\eps}(t,\cdot)\ast\varphi_{\delta}(x)|^{2}\diff x \diff t\\
&+4\int_{0}^{T-h}\int_{\Td}|f_{\eps}(t+h,x)-f_{\eps}(t+h,\cdot)\ast\varphi_{\delta}(x)|^{2}\diff x \diff t\\
&+4\int_{0}^{T-h}\int_{\Td}|f_{\eps}(t+h,\cdot)\ast\varphi_{\delta}(x)-f_{\eps}(t,\cdot)\ast\varphi_{\delta}(x)|^{2}\diff x \diff t.
\end{align*}
When we apply limit $\lim_{h\to 0} \limsup_{\varepsilon\to0}$, the first and second term vanish due to \eqref{item2compactness_2}.
It remains to study the third term. For this, suppose first that $f_{\varepsilon}(t,\cdot)$ is smooth in the time variable. Then,
\begin{align*}
\int_{0}^{T-h}&\int_{\Td}|f_{\eps}(t+h,\cdot)\ast\varphi_{\delta}(x)-f_{\eps}(t,\cdot)\ast\varphi_{\delta}(x)|^{2}\diff x \diff t=\int_{0}^{T-h}\int_{\Td}\left|\int_{t}^{t+h}\p_{t}f_{\eps}(s,\cdot)\ast\varphi_{\delta}(x) \diff s\right|^{2}\diff x \diff t.
\end{align*}
Now, we can estimate the convolution as follows
$$
\p_{t}f_{\eps}(s,\cdot) \ast \varphi_{\delta}(x) = \int_{\Td} \p_{t}f_{\eps}(s,y) \, \varphi_{\delta}(x-y) \diff y \leq \|\partial_t f_{\varepsilon}(s,\cdot)\|_{H^{-k}(\Td)} \, \| \varphi_{\delta}(x-\cdot)\|_{H^k(\Td)}.
$$
Using this and applying invariance in space of the $H^k(\Td)$ norm we obtain
$$
\int_{0}^{T-h}\int_{\Td}\left|\int_{t}^{t+h}\p_{t}f_{\eps}(s,\cdot)\ast\varphi_{\delta}(x) \diff s\right|^{2}\diff x \diff t = h^2 \, \| \varphi_{\delta}\|_{H^k(\Td)} \int_{0}^{T-h}\left|\frac{1}{h}\int_{t}^{t+h}\|\partial_t f_{\varepsilon}(s,\cdot)\|_{H^{-k}(\Td)} \diff s\right|^{2} \diff t.
$$
Applying Jensen's inequality we obtain
$$
h^2 \, \| \varphi_{\delta}\|_{H^k(\Td)} \int_{0}^{T-h}\left|\frac{1}{h}\int_{t}^{t+h}\|\partial_t f_{\varepsilon}(s,\cdot)\|_{H^{-k}(\Td)} \diff s\right|^{2} \diff t 
\leq
Th \, \| \varphi_{\delta}\|_{H^k(\Td)} \|\partial_t f_{\varepsilon}(s,\cdot)\|_{L^2(0,T;H^{-k}(\Td))}^2. 
$$
Using that $ \|\varphi_{\delta}\|_{H^{k}(\Td)}\le\frac{C}{\delta^{k+d/2}}$ we finally obtain
$$
 \int_{0}^{T-h}\int_{\Td}|f_{\eps}(t+h,\cdot)\ast\varphi_{\delta}(x)-f_{\eps}(t,\cdot)\ast\varphi_{\delta}(x)|^{2}\diff x \diff t
 \leq CT \frac{h}{{\delta^{k+d/2}}} \, \|\partial_t f_{\varepsilon}(s,\cdot)\|_{L^2(0,T;H^{-k}(\Td))}^2.
$$
Now, if $f_{\varepsilon}(t,x)$ is not smooth in time, we extend it with $f_{\varepsilon}(0,x)$ for $t<0$, $f_{\varepsilon}(T,x)$ for $t>T$ and apply usual regularization to obtain the same estimate. Hence, if we choose $h(\delta) = \delta^{2k+d}$ we conclude
\begin{equation*}
   \lim_{h\to 0} \limsup_{\varepsilon\to0}  \int_{0}^{T-h}\int_{\Td}|f_{\eps}(t+h,x)-f_{\eps}(t,x)|^2\diff x \diff t\le \theta(h).
\end{equation*}

Combined with the compactness in space~\eqref{item2compactness} and the Fréchet-Kolmogorov theorem we obtain the compactness of $\{f_{\eps}\}$ in $L^{2}((0,T)\times \Td)$. 
\end{proof}

\subsection*{Acknowledgements}
JAC was supported by the Advanced Grant Nonlocal-CPD (Nonlocal PDEs for Complex Particle Dynamics: Phase Transitions, Patterns and Synchronization) of the European Research Council Executive Agency (ERC) under the European Union’s Horizon 2020 research and innovation programme (grant agreement No. 883363). JAC was also partially supported by the EPSRC grant numbers EP/T022132/1 and EP/V051121/1. JS was supported by National Science Center, Poland through project no. 2019/35/N/ST1/03459.

\bibliographystyle{abbrv}
\bibliography{fastlimit}
\end{document}